\documentclass[12pt,letterpaper,oneside,openright]{book}
\usepackage{amscd,amsfonts,amsmath,amssymb,amstext,amsthm}
\usepackage[margin=1in]{geometry}  
\usepackage{graphicx}              
\usepackage{index}
\usepackage{newlfont}

\theoremstyle{plain}
\newtheorem{thmI}{Theorem}[chapter]

\newtheorem{thm}{Theorem}[section]

\newtheorem{cor}[thm]{Corollary}
\newtheorem{lem}[thm]{Lemma}
\newtheorem{prop}[thm]{Proposition}

\theoremstyle{definition}
\newtheorem{Q}{Question}
\newtheorem{defn}[thm]{Definition}

\newtheorem{rem}[thm]{Remark}
\newtheorem{rems}[thm]{Remarks}

\theoremstyle{Example}
\newtheorem{ex}[thm]{Example}

\begin{document}
\nocite{*}

\title{\textsc{\Large Combinatorial geometry of flag domains in $G/B$}}

\author{\textsc{\Large Dissertation}\\ \ \\ \ \\
zur Erlangung des Doktorgrades\\
 der Naturwissenschaften\\
an der Fakult\"at f\"ur Mathematik\\
der\\ \ \\
\textsc{\Large Ruhr-Universit\"at Bochum} \\ \ \\ \ \\
vorgelegt von\\ \ \\
\textsc{\Large Faten Said Qarmout}\\
n\`ee Abu-Shoga \\ \ }

\date{im \\ August 2017}

\sloppy
\parindent0ex
\maketitle
\pagenumbering{roman}


\null \vspace{\stretch{1}}
\begin{flushright}
To everyone who wished the best for me.~~~~~~~~~~~~~~~~~~~~~~~~~~~~~~~~~
\end{flushright}
\vspace{\stretch{5}}\null

 \newpage
\chapter*{Acknowledgement}
"Be a scientist, if you can not,  be a student, if you can not, love scientists, if not, don't hate them"

\bigskip\noindent

It is a great honor to thank everyone who supported me in my scientific career and helped me overcome the obstacles and challenges that have always faced me.
I would like to express my thanks to God Almighty, who has enlightened my life and made me choose science to be my path.

\bigskip\noindent

Undertaking this Ph.D. has been a truly life-changing experience for me and it would not have been possible to do without the support and guidance that I received from my supervisors Prof. Dr. Peter Heinzner and Prof. Dr. Alan Huckleberry. I take the opportunity to send rivers of thanks to them for giving me the golden opportunity to come to Germany and complete my Ph.D. study.

\bigskip\noindent
I would like to express the deep appreciation of Professor Peter Heinzner for granting me first knowledge of the theory of Lie group and complex geometry. Thank you very much for directing many useful and interesting seminars on transformation groups and for making my time as a graduate student at the Ruhr University of Bochum very fruitful. I would also like to thank you for your support and assistance in passing all the challenges that I have encountered during my studies, and forced me to work hard and to complete my research.

\bigskip\noindent

I would like to express the deep appreciation to Professor Alan Huckleberry for his guidance and encouragement during my studies.  He has been actively interested in my work and has always been available to advise me. Thank you for forcing me to work more and more and think in different ways. It would have been very hard for me to understand combinatorial geometry and complete my research without your observations and advice and your deep knowledge of cycle spaces.

\bigskip\noindent

Besides my advisers, I would like to thank PD. Dr. Stéphanie Cupit-Foutou who supported me in mathematical discussions and also fortunately by talking about things other than my research. Also I would like to thank all my colleagues in the Transformation Groups seminar.

\bigskip\noindent

I gratefully acknowledge the funding received towards my Ph.D. from the Islamic Development Bank-The Merit Scholarship- during the first $3$ years in my study. I am also grateful to the funding received from the Ruhr University of Bochum.

\bigskip\noindent

I would also like to thank my parents for their wise advice and sympathetic ear. They were always with me.  I send my thanks to the candles that illuminate my life, my parents, my brother, my sister, my husband and my children, Ahmed and Mohammed.


\tableofcontents{}

\pagenumbering{arabic}

 \setcounter{page}{1}

\chapter{Introduction}
\section{Preliminaries}
Let $G$ denote a complex semisimple Lie group with real form $G_0$ and $Z:=G/Q$
be a $G$-flag manifold. It is known that $G_0$ has only finitely many orbits in $Z$ and
that there is a unique closed orbit $\gamma^{cl}$. The open $G_0$-orbits are referred to as
\emph{flag domains} $D$.  Every maximal compact subgroup $K_0$ of $G_0$ has a unique
orbit $C_0$ in $D$ which is a complex submanifold.  This is the unique $K_0$-orbit  of minimal
dimension in $D$.  If $K$ is the corresponding (reductive) subgroup of $G$, then it can also be
characterized as the unique $K$-orbit which is contained in $D$.  Having fixed $K_0$, we
refer to $C_0$ as the associated \emph{base cycle}. For the above and other background information see (\cite{FHW} and \cite{W1}).

\bigskip\noindent
The main goal of our work can be formulated as that of determining the class of $C_0$ in the
(topological) cohomology ring of $Z$.  If $B$ is any Borel subgroup of $G$
and $r:=dim_\mathbb{C}C_0$, this amounts to determining the intersection numbers
of the $r$-codimensional $B$-Schubert varieties $S$ with $C_0$. Of course it is possible that
this intersection is empty.

\bigskip\noindent
 The original motivation
for this project was another:  If $B$ is maximally real in the sense that it fixes a point in $\gamma^{cl}$,
or equivalently if it contains the $A_0N_0$-factor of an Iwasawa-decomposition $G_0=K_0A_0N_0$,
and $S$ is a $B$-Schubert variety (we refer to such as Iwasawa-Schubert varieties) of complementary dimension to the cycle $C_0$, then $S\cap D$ is
contained in the open $B$-orbit $\mathcal{O}$ in $S$ and the intersection $\mathcal{O}\cap C_0$
is finite and transversal at each of its points (see \cite{FHW}).  In fact each component $\Sigma$ of $\mathcal{O}\cap D$
intersects $C_0$ in exactly one point $z_\Sigma $ with $\Sigma =A_0N_0.z_\Sigma $.  This in turn gives
rise to an optimal Barlet-Koziarz trace-transform from holomorphic functions on $\mathcal{O}\cap D$
to holomorphic functions on the appropriately defined space $\mathcal{C}(D)$ of cycles (see \cite{BK}, \cite{FHW}).  Thus, although the determination of
the $r$-codimensional Iwasawa-Schubert varieties which have non-empty intersection with $D$, along with their
points of intersection with $C_0$, is apparently a problem of a combinatorial nature, such information has complex analytic
significance.

\bigskip\noindent
Since $S \cap D \neq \emptyset$ implies that $S \cap C_0 \neq \emptyset$, the first goal of this project is to determine which Schubert varieties $S$ have nonempty intersection with $C_0$. After doing so, we then describe this intersection in the case where $S$ is of complementary dimension to $C_0$. The Schubert varieties are determined by the elements of the Weyl group $W_I$ of a distinguished maximal torus $T_I$ in the Borel subgroup $B_I$ which fixes a certain base point in $\gamma_{c\ell}$. These Schubert varieties are denoted by $S_w$ where $w \in W_I$. Consequently our results in are formulated in terms of conditions on elements $w$ in the Weyl group $W_I$.

\bigskip\noindent
In this thesis we make several restrictions. First, we only consider the case where $G$ is a classical group,
i.e., associated to one of the Lie algebras $A_n$, $B_n$, $C_n$ or $D_n$.  The case of $A_n$ (and all of its real
forms) was handled by Brecan (\cite{Bre},\cite{Bre2}). Our work here is therefore devoted to the cases of $B_n$, $C_n$ and $D_n$. Furthermore, we restrict to the setting where $Z=G/B$ is the $G$-manifold of full flags.  This is reflective of our work with the advantage that technical complications are minimized. The case of a general $G$-flag manifold $Z=G/Q$ is handled in author paper (\cite{F-S}).

\bigskip\noindent

 \section{The structure of the thesis}
As indicated above this thesis is devoted to study the combinatorial geometry of the flag domains in $G/B$ of the real forms $SP(2n,\mathbb{R}),SO^*(2n),$ $SO(p,q)$ and $SP(2p,2q)$, where $B$ is Borel subgroup. We study the intersection between the base cycles of the flag domains and the Schubert varieties. In the cases of $SP(2n,\mathbb{R})$ and $SO^*(2n)$ a remarkable
unicity appears so that an exact description of the Weyl elements can be given. Analogous to the case of $SU(p,q)$ in the
work of Brecan, except for a few interesting special cases, for $SO(p,q)$ and $SP(2p,2q)$ these elements are described by
algorithms.  In all cases the number of intersection points with $C_0$ is explicitly computed.

\bigskip\noindent

The work here is organized in six chapters which are briefly described as follows:

\bigskip\noindent

\textbf{Chapter $2$:} In this chapter we explain the formulas for all real forms $G_0$ which arise in this thesis. In each case two base points are introduced: a base point $F_I$ in $\gamma_{c\ell}$ and a base point $F_S$ in a certain standard cycle. These two base points are full flags associated to bases which define the maximal Tori $T_I$ and $T_S$ respectively. The Borel subgroup $B_I$ and the Weyl group $W_I$ are described in terms of the ordered basis which defines $F_I$. Also, we define a special way of writing the Weyl group elements for the Lie groups of type $B_n,C_n$ and $D_n$ in terms of flags. In each case the flag domains and base cycles are described by signature invariants.

\bigskip\noindent

\textbf{Chapter $3$:} Here we prove two general fixed point theorems:
\begin{thmI}
For every $w \in W$, if the intersection $S_w \cap C_0 $ is non-empty, then it contains a $T_S$-fixed point.
\end{thmI}
\begin{thmI}
If $S_w$ is of complementary dimension to $C_0$, then $S_w \cap C_0 \subset Fix(T_S)$.
\end{thmI}

\bigskip\noindent
In the following chapters we focus on the goal of understanding the combinatorial geometry of the topological class $[C_0]$ by attempting to find precises formulas which answer the following question in each of the cases $B_n$, $C_n$ and $D_n$.
\begin{Q}
 What are the conditions on the Weyl elements $w$ which parametrize the Iwasawa-Schubert Varieties $S_w$ with non-empty intersection with the base cycle $C_0$?
\end{Q}
\begin{Q}
 What are the points of intersection? How many of these points do we have?
\end{Q}

\bigskip\noindent

\textbf{Chapter $4$:} This chapter is concerned with the case of the real form $SP(2n,\mathbb{R})$. In this case there is a certain Weyl element, which we refer to as a "super generous permutation". Since we are considering flag domains in $G/B$, all base cycles have the same dimension. In terms of a super generous permutation the main result of this chapter can be stated as follows.
\begin{thmI}
If $S_w$ is an Iwasawa-Schubert variety of complementary dimension to a base cycle $C_0$ and
$S_w\cap C_0\not=\emptyset $, then $w$ is the super generous permutation. Furthermore, the intersection $S_w\cap C_0$ with every base cycle consists of just one point.
\end{thmI}

\bigskip\noindent
\textbf{Chapter $5$:} This chapter is devoted to study the real form $SO^*(2n)$ which is analogous to that for the case of the real form $SP(2n,\mathbb{R})$. Again a uniqueness theorem of the above type is valid, but in this case the results depend on $n$ being even or odd.

\bigskip\noindent
\textbf{Chapter $6$:} The real form of this chapter is $SO(p,q)$. The results are stated in terms of algorithms (See definitions \ref{de1}, \ref{de2} and Corollary \ref{cor12.4})
; in fact it seems impossible to avoid this. Also it would seem to be difficult to find a concrete formula for the intersection points  in the general case. We do however explicitly compute the numbers of intersection points (see Theorem \ref{thm5.18} and Corollary \ref{cor12.4}).

\bigskip\noindent
\textbf{Chapter $7$:} This chapter is devoted to the real form  $SP(2p,2q)$. Our work here is similar to that for $SO(p,q)$. Again the results are stated in terms of algorithms (Corollary \ref{cor6.17}, Definition \ref{majorD1} and Definition \ref{majorD2}) and again it would seem to be difficult to find a concrete formula for the intersection points  in the general case. We do, however, explicitly compute the numbers of intersection points (see Theorem \ref{8.4.5}).

\chapter{Basic ingredients}
The Lie subgroup $G_0$ is a real form of the classical group $G$ which is assumed to be
of type $B_n$, $C_n$ or $D_n$, and $Z$ is in each case the associated manifold
of full flags.  In each case we introduce appropriate bases of the relevant vector
space so that $G$, $G_0$ and the compact and reductive groups, $K_0$ and $K$,
are described in matrix terms which are useful for our purposes.  Regarded as flags,
these bases define base points in $Z$ along with Borel subgroups with distinguished
maximal tori.  Of particular importance is the base point $F_I$ in the closed $G_0$-orbit,
because the Iswasawa-Schubert cells which will be considered  are orbits of
its isotropy group $B_I$ of points $w(F_I)=:F_w$, where $w$ runs through the Weyl group
$W_I$ defined by the distinguished maximal torus $T_I$ in $B_I$.   Finally, in each of
the cases we describe the flag domains $D$ in $Z$ by signature invariants and the
base cycle $C_0$ in terms of splittings of $V$.
\section{Introduction to the flags in $Z$}
\bigskip\noindent

Let $G$ be a semisimple Lie group of type $B_n,C_n$ or $D_n$, defined in each case by a non-degenerate bilinear form $b(v,w)$.
 A maximally $b$-isotropic full flag $F$ with respect to $b$ is a sequence of $(m+1)$-vector spaces
$$F=(\{0\}=V_0 \subset V_1 \subset ... \subset V_{m}=\mathbb{C}^{m})$$
such that $dim~V_i=i,~for~all ~ 0\leq i\leq m$ and $V_{m-i}=V_{i}^{\bot} ,~\text{for} ~ 1\leq i\leq \left\lfloor  \frac{m}{2}\right\rfloor$,  where $m=2n$ or $2n+1$. In particular $V_i$ is isotropic for $1\leq i\leq n$. The set of all maximally isotropic flags is denoted by $Z$.
The manifold structure of $Z$ arises from the fact that the Lie group $G$ acts transitively on it with $Z=G/B$ where $B$ is the stabilizer of any particular maximally $b$-isotropic flag in $Z$.

\bigskip\noindent

In all cases $\sigma :\mathbb{C}^{2n}\to \mathbb{C}^{2n}$, $u\mapsto \bar{u}$, is standard complex conjugation and $G_0$ is defined to be the fixed set of $ \sigma $,i.e,
$$
G_0=\{g \in G: \sigma(g)=g\}.
$$

The signature of a flag $F =(0\subset V_1 \subset ... \subset V_{m}) \in Z$ with respect to $h$ consists of three sequences
$$a:0 \leq a_1 \leq a_2 \leq ... \leq a_{m} $$
 $$b:0 \leq b_1 \leq b_2 \leq ... \leq b_{m} $$
$$d:0 \leq d_1 \leq d_2 \leq ... \leq d_{m} $$
 with $sign(V_i)=(a_i,b_i,d_i)$ where $a_i$ (resp. $b_i$) denotes the dimension of a maximal negative (positive) subspace and $d_i$ the degeneracy $V_i$ of the restriction of $h$ to $V_i$ and $a_i+b_i+d_i=i~~\forall 1\leq i \leq 2n$. The 3-tuple $(a,b,d)$ is called the signature of the flag $F$. If $d\not=0$, we refer to $F$ as being non-degenerate and write $sign(F)=(a,b)$.\\


\section{The real form $Sp(2n,\mathbb{R})$}
\subsection{Preliminaries}\label{Sp}
In order to define the symplectic group $G=Sp(2n,\mathbb{C})$ we introduce $J_n$ to be $$J_n =\left(\begin{array}{ccccc}  0&0&\cdots &0 & 1\\0&0&\cdots &1 & 0\\\cdots &\cdots & \cdots&\cdots& \cdots\\0&1&\cdots &0 & 0\\
1 &0 &\cdots&0&0  \end{array}\right)$$
 and  $b(v,w)=v^t \left(\begin{array}{cc} 0 & J_n\\  -J_n & 0 \end{array}\right) w.$
Then $$Sp(2n,\mathbb{C}):=\{g \in GL(2n,\mathbb{C}):g^t \left(\begin{array}{cc} 0 & J_n\\  -J_n & 0 \end{array}\right)  g= \left(\begin{array}{cc} 0 & J_n\\  -J_n & 0   \end{array}\right) \}.$$  The Hermitian form $h(v,w)$ of the signature $(n,n)$ is defined by $$h(v,w)=b(v,\sigma (w))= \sum_{i=1}^n v_i \bar{w}_i -\sum_{i=n+1}^{2n} v_i \bar{w}_i\,$$
and the real form $G_0=Sp(2n,\mathbb{R})=Sp(2n,\mathbb{C}) \cap U(n,n)$.

\bigskip\noindent

Here a Cartan involution $\theta$ is given by
$$\theta:g\longrightarrow \left(\begin{array}{cc} I_n & 0\\ 0 & -I_n \end{array}\right) g  \left(\begin{array}{cc} I_n & 0\\ 0 & -I_n \end{array}\right)$$
and the maximal compact subgroup associated to it is
$$K_0=\{\left(\begin{array}{cc} U & 0\\ 0 & J_n (U^t)^{-1} J_n \end{array}\right):U \in U(n,\mathbb{C}) \}\approx U(n)$$

\subsection{Base points}

The standard basis $(e_1,e_2,...,e_{2n})$ of $\mathbb{C}^{2n}$ defines the standard maximal torus $T_S$ as the subgroup of diagonal matrices, i.e,
$$T_S=\{g \in  G_0:g=diag(t_1,t_2,...,t_n,-t_n,...,-t_2,-t_1), t_i \in \mathbb{C} \}$$
and the standard Borel subgroup $B$ of upper triangular matrices in $Sp(2n,\mathbb{R})$.

\bigskip\noindent

The standard basis defines a base point $F_S$ in $Z$ as the associated flag
$$
<e_1>\subset <e_1,e_2>\subset ... \subset <e_1,...,e_{2n}>
$$
Reorderings on this basis will  determine the flag domains, the base cycles and the intersection points.

\bigskip\noindent

In order to define a base point in $\gamma^{c\ell}$ we introduce the basis
\begin{equation}\label{basisAA}
e_1-e_{2n},e_2-e_{2n-1}...,e_n-e_{n+1},e_n+e_{n+1},...,e_2+e_{2n-1},e_1+e_{2n}
\end{equation}
with $T_I$ being the maximal torus of diagonal matrices in $G$ and $B_I$ the Iwasawa-Borel subgroup of $G$ which fixes the associated flag $F_I$.

\subsection{Weyl group}\label{Weyl}
If $(r_1,\ldots ,r_n,s_1,\ldots,s_n)$ is a basis with $b(r_i,r_j)=0$, $b(s_i,s_j)=0$ and $b(r_i,s_j)=\delta_{ij}$, then we regard $V=U\oplus U^*$ with $U=Span\{r_1,\ldots,r_n\}$ and $U^*=Span\{s_1,\ldots,s_n\}$.  For such a basis $b$ is in \emph{canonical form}. If $T$ is the maximal torus in $G$ defined by such a basis, then the Weyl group $W(T)$ acts on $T$ by certain permutations of the basis.  These are of the form $w=\pi E$ where $\pi $ is an arbitrary permutation of $(r_1,\ldots ,r_n)$ and $E$ exchanges an arbitrary number of the $r_i$ with the corresponding dual vectors $s_i$. If we have a maximally isotropic flag associated to some permutation of the ordered basis $(r_1,\ldots,r_n,s_n,\ldots,s_1)$, then in the first $n$ positions of the flag the above condition holds: if $r_i$ appears, then $s_i$ does not, and vice versa. Now the full flag is determined by the first $n$ positions; regard it as a permutation of the $r_i$ and, if $s_i$ occurs instead of an $r_i$, we mark that
spot with a minus sign.  In this way $W(T)$ is the semidirect product $ {\displaystyle S_{n}\ltimes \mathbb{Z}_2^{n}}$ of a symmetric group and a normal elementary abelian 2-subgroup. If the basis in question is that in (\ref{basisAA}) above, then we denote this Weyl group by $W_I$.  Analogously, if the basis in question is the standard basis $e_1,...,e_{2n}$, then we denote this Weyl group by $W_{T_S}$.

 \bigskip\noindent

In our particular case there are two maximal tori, the maximal torus $T_I$ defined by \ref{basisAA} and the maximal torus $T_S$ defined by $(e_1,e_2,...,e_{2n})$.
Recall that these two maximal tori $T_I, T_S$ are conjugate, and a conjugation induces an isomorphism $\psi$ of the associated Weyl groups.  In this case we define a bijective map $\psi(\mp (i))=\pm i,$ for all $1\leq i\leq n$ for later use.

\subsection{Flag domains and base cycles}\label{flag}

Since in every flag domain $D$ there is a unique
closed $K$-orbit, the associated base cycle $C_0$, in order to parametrize the flag domains it is enough it is enough to parametrize the
closed $K$-orbits, (see \cite{FHW} , $\S4.2$ and $\S4.3$).   For this we choose a base point $F_S\in Fix(T_S)$. It follows that the set
$\mathcal{C}$ of closed $K$-orbits can be identified with the Weyl group orbit
$W_G(T_S).F_S$ (see \cite{FHW}, Corollary 4.2.4).  As we have observed above, $W_G(T_S)\cong S_n\ltimes \mathbb{Z}_2^n$
where the $S_n$ factor can be identified with $W_K(T_S)$. In fact it is exactly the
stabilizer of the $K.F_S$ in $\mathcal{C}$.  Consequently $\mathcal{C}$ can be identified with
$\mathbb{Z}_2^n.F_S$. We regard an element of $\mathbb{Z}_2^n$ as a vector $\alpha $ of length $n$ consisting of plus- and minus-signs, and thus $\mathcal{C}$ can be identified with the set of such vectors. In concrete terms the $K$-orbits (resp. $G_0$-orbits ) of
base points $(\pm e_1,\ldots ,\pm e_n)$ are exactly the base cycles (resp. flag domains) in $Z$. The respective flag domains are denoted by $D_\alpha $.

\bigskip\noindent

Observe that $\alpha $ defines a non-degnerate signature $(a,b)$ and that $D_\alpha\subset D_{a,b}:=\{F:sign(F)=(a,b)\}$.
Conversely, given $F$ with $sign(F)=(a,b)$ we may choose a basis $(v,w):=(v_1,\ldots ,v_n,w_n,\ldots w_1)$ which defines
$F$ and which has the following properties:
\begin {enumerate}
\item
 $b$ is in canonical form.
\item
$(v,w)$ is $h$-orthogonal.
\item
$\Vert v_j\Vert^2_h=\pm 1$ with $sgn(\Vert v_j\Vert_h )=\alpha_j$.
\item
$\Vert w_j\Vert^2_h=-\Vert v_j\Vert^2_h$.
\end {enumerate}
If  $sign(F)=sign(\hat{F})$, then we choose bases $(v,w)$ and $(\hat{v},\hat{w})$ for  $F$ and $\hat{F}$, respectively, and
note that the transformation which takes the one basis to the other is in $G_0$.  This argument shows that
$D_{a,b}=D_\alpha $.

\bigskip\noindent

Denote $E_+=<e_1,...,e_n>$, and $E_-=<e_{n+1},...,e_{2n}>$, where $\mathbb{C}^{2n}=E_+ \oplus E_-$.
\begin{prop}\label{Base cycle}
For a fixed open orbit $D_{\alpha}$, the base cycle $C_{\alpha}$ is the set
$$C_{\alpha} =\{F \in Z : V_i= (V_i \cap E^-) \oplus (V_i \cap E^+) , 1 \leq i \leq 2n\}.$$
\end{prop}
\begin{proof}
Firstly, it was shown above that there is a base point $F_\alpha $ with the splitting condition. Since $K_0$ is the product $K_0=K_0^+ \times K_0^-$, it acts transitively on all such flags. Let $F \in C_{\alpha}$ where $V_i^-=V_i \cap E^-$ and $V_i^+=V_i \cap E^+$. Define $F^+$ and $F^-$ to be the sets of maximal isotopic flags in $E^+$ and $E^-$ respectively. Recall that $K_0=K_0^+ \times K_0^-$ where $K_0^+$ and $K_0^-$ act transitively on the sets $F^+$ and $F^-$ respectively, which implies that $K_0$ acts transitively on the set of maximal isotropic flags obtained by put flags from $F^+$ and $F^-$ in away such that the new flag has signature $\alpha$. Hence $C_{\alpha}$ is a complex manifold. But $K_0$ has a unique orbit in $D_{\alpha}$ which is a complex manifold. Therefor $C_{\alpha}$ is the base cycle.
\end{proof}

\section{The real form $SO^*(2n)$}
\subsection{Preliminaries}
Let $G=SO(2n,\mathbb{C})$  be the special orthogonal group defined by the non-degenerate complex bilinear form $b:\mathbb{C}^{2n}\longrightarrow \mathbb{C}^{2n}$ where $b(v,w)=\sum_{i=1}^{2n} v_{2n-i+1}w_i$. The real form $G_0=SO^*(2n)$ is the subgroup of elements of $G$ which leave invariant the Hermitian form defined by $$h(v,w)=b(v,\bar{w})=\sum_{i=1}^n v_i \bar{w}_i -\sum_{i=n+1}^{2n} v_i \bar{w}_i .$$ This real form is the fixed point set of the involution $\tau:SO(2n,\mathbb{C})\longrightarrow SO(2n,\mathbb{C})$ defined by
$$\tau(g)=\left(\begin{array}{cc} 0 & J_n\\  J_n & 0 \end{array}\right)  \bar{g}^{-t} \left(\begin{array}{cc} 0 & J_n\\  J_n & 0 \end{array}\right) ,~~~where~J_n=\left(\begin{array}{ccccc}  0&0&\cdots &0 & 1\\0&0&\cdots &1 & 0\\\cdots &\cdots & \cdots&\cdots& \cdots\\0&1&\cdots &0 & 0\\
1 &0 &\cdots&0&0  \end{array}\right).$$
 Note that $SO^*(2n)\cong SO(2n,\mathbb{C}) \cap SU(n,n)$.\\
A Cartan involution $\theta:G_0\longrightarrow G_0$ is given by
$$\theta:g\longrightarrow \left(\begin{array}{cc} I_n & 0\\ 0 & -I_n \end{array}\right) g  \left(\begin{array}{cc} I_n & 0\\ 0 & -I_n \end{array}\right)$$
and the maximal compact subgroup associated to it is
$$K_0=\{\left(\begin{array}{cc} U & 0\\ 0 & -J_n (U^t)^{-1} J_n \end{array}\right): U \in U(n) \}\approx U(n)$$

\subsection{Base points}

By an argument which is completely analogous to that in the case of $Sp(2n,\mathbb{R})$ the standard basis $(e_1,e_2,...,e_{2n})$ of $\mathbb{C}^{2n}$ defines the standard maximal torus $T_S$, i.e,
$$T_S:=\{g \in  G_0:g=diag(t_1,t_2,...,t_n,-t_n,...,-t_2,-t_1), t_i \in \mathbb{C} \}$$
and the standard Borel subgroup $B$ of upper triangular matrices in $SO^*(2n)$.

\bigskip\noindent

The standard basis defines a base point $F_S$ in $Z$ as the associated flag
$$
<e_1>\subset <e_1,e_2>\subset ... \subset <e_1,...,e_{2n}>
$$
Reorderings on this basis will  determine the flag domains, the base cycles and the intersection points.

\bigskip\noindent

In order to define a base point in $\gamma^{c\ell}$ we introduce the basis
\begin{equation}\label{basis n even}
e_1+ie_{n+1},e_2+ie_{n+2}...,e_n+ie_{2n},e_n-ie_{2n},...,e_2-ie_{n+2},e_1-ie_{n+1},~~ \text{if~n~is~even};
\end{equation}
or
\begin{equation}\label{basis n odd}
e_1+ie_{n+1},...,e_n+ie_{2n},e_{\frac{n+1}{2}},e_{n+\frac{n+1}{2}},e_n-ie_{2n},...,e_1-ie_{n+1}, ~~\text{if~n~is~odd}.
\end{equation}
with $T_I$ being the maximal torus of diagonal matrices in $G$ and $B_I$ the Iwasawa-Borel subgroup of $G$ which fixes the associated flag $F_I$.
\subsection{Weyl groups, flags domains and base cycles}\label{Weyl2}
By arguments which are essentially the same as those in $\S \ref{Weyl}$, if
$T$ is a maximal torus defined by a basis
of $(r,s)$-type, then $W(T)=S_n\ltimes \mathbb{Z}_2^{n-1}$ where now the sign-change vector $\alpha $ has only an even number of minuses.
 As in the case of $Sp(2n,\mathbb{R})$, to describe the flag domains and their $K_0$-base cycles we use the Weyl group $W(T_S)$ of the standard torus $T_S$. The base point for $\alpha $ consisting only of pluses is the flag $F_S$ is defined by the standard basis.  Then
all base points are
$F_\alpha :=w(F_S)$ where $w\in \mathbb{Z}_2^{n-1}$ is associated to the sign-change vector $\alpha $.  As before the base cycles and flag domains are the orbits $K.F_\alpha $ and $G_0.F_\alpha $, respectively, and $\alpha $ defines a non-degenerate sign with $sign(F_\alpha)=(a,b)$ and then $D_\alpha =D_{a,b}$.

\bigskip\noindent

In our particular case there are two maximal tori, the maximal torus $T_I$ defined by (\ref{basis n even}) or (\ref{basis n odd}) and the maximal torus $T_S$ defined by $(e_1,e_2,...,e_{2n})$.
Recall that these two maximal tori $T_I, T_S$ are conjugate, and a conjugation induces an isomorphism $\psi$ of the associated Weyl groups.  In this case we define a bijective map $\psi(\mp (i))=\pm i,$ for all $1\leq i\leq n-1$ and $\psi(\pm n)=\pm n$.
.

\bigskip\noindent

Denote $E_+=<e_1,...,e_n>$, and $E_-=<e_{n+1},...,e_{2n}>$, where $\mathbb{C}^{2n}=E_+ \oplus E_-$. Again, by an argument which is completely  analogous to Proposition \ref{Base cycle} of the case $SP(2n,\mathbb{R})$, the base cycle $C_0$ of the flag domain $D_{\alpha}$ of $SO^*(2n)$ is the set
$$C_0 =\{F \in Z : V_i= (V_i \cap E^-) \oplus (V_i \cap E^+) , 1 \leq i \leq 2n\}.$$
which is convenient to use the orbit $\mathbb{Z}_2^{n-1}.F_S$ of the base point as the set of intersection point of $S_w \cap D_{\alpha}$. The intersection dimensions are determined by $\alpha $ (resp. $(a,b)$). See section \ref{flag} for more details.

\section{The real form $SO(p,q)$}

\subsection{Preliminaries }

Consider the semisimple Lie group $G=SO(m,\mathbb{C})$ where $m=2n$ or $2n+1$. In this case it is convenient to choose the bilinear form $b$ depending on the $p$ and $q$.  If $p$ or $q$ is even, then we choose it in the usual way:
$$b(v,w)=-\sum_{i=1}^{q} v_i w_i+ \sum_{i=q+1}^{m} v_i w_i$$
Let $\sigma$ be the complex conjugation $\sigma (v) = \bar{v}$, then the Hermitian form $h(v,w)$ of signature $(p,q)$ which defines the real form is defined by $$h(v,w)=b(v,\sigma (w))= -\sum_{i=1}^q v_i \sigma (w_i) +\sum_{i=q+1}^{m} v_i \sigma(w_i).$$
  If both $p$ and $q$ are odd, then the complex bilinear form is
$$b(v,w)=-\sum_{i=1}^{q-1} v_i w_i+(v_q w_{q+1}+v_{q+1} w_{q}) +\sum_{i=q+2}^{m} v_i w_i$$  and the Hermitian form $h(v,w)$ is defined by $$h(v,w)= -\sum_{i=1}^q v_i \bar{w}_i +\sum_{i=q+1}^{m} v_i \bar{w}_i.$$
The real form is $G_0=SO(p,q)=SO(2n,\mathbb{C}) \cap SU(p,q)$.
\bigskip\noindent

A Cartan involution $ \theta:\mathfrak{so}(p,q) \longrightarrow \mathfrak{so}(p,q)$ in the Lie algebra level is given by $\theta(g)=-g^t$, and the maximal compact subgroup associated to it is $K_0 := S(O(p) \times O(q))\subset S(U(p) \times U(q))$.

\subsection{Base points}\label{basepoint}
 In the following we define bases depending on $p$ and $q$ as well as the positive and negative spaces which will be essential for understand the base cycle.  The fundamental maximal torus $T_S$ is defined in each case by requiring that each basis vector is a $T_S$-eigenvector.

\begin{description}
	\item[$\bullet$ If $m=2n$ and $q$ is even: ] The ordered $b$-isotropic basis is
 \begin{eqnarray}\label{basis11}
e_1+ie_2,e_3+ie_4,......,e_{2n-1}+ie_{2n},e_{2n-1}-ie_{2n},......,e_1-ie_2,
 \end{eqnarray}
  where $E^-:=<e_1+ie_2,e_3+ie_4,...,e_{q-1}+ie_{q},e_1-ie_2,e_3-ie_4,...,e_{q-1}-ie_{q}>$ and \\$E^+:=<e_{q+1}+ie_{q+2},...,e_{2n-1}+ie_{2n},e_{q+1}-ie_{q+2},..., e_{2n-1}-ie_{2n}>$.
	
	\item[$\bullet$ If $m=2n$ and $q$ is odd: ] The ordered $b$-isotropic basis is
 \begin{eqnarray}\label{basis22}
e_1+ie_2,e_3+ie_4,......,e_{2n-1}+ie_{2n},e_q,e_{q+1},e_{2n-1}-ie_{2n},......,e_1-ie_2,
 \end{eqnarray}
 where  $E^-=<e_1+ie_2,e_3+ie_4,...,e_{q},e_1-ie_2,e_3-ie_4,...,e_{q-1}-ie_{q}>$ and\\ $E^+=<e_{q+1}, e_{q+2}+i e_{q+3},...,e_{2n-1}+ie_{2n},e_{q+2}-ie_{q+3},..., e_{2n-1}-ie_{2n}>$.\\
Note that the vectors $e_q$ and $e_{q+1}$ are isotropic and $b(e_q,e_{q+1})=-1$. Moreover $h(e_3,e_3)=-1$ and $h(e_4,e_4)=1$ which means that
$e_q \in E^-$ and $e_{q+1} \in E^+$.
\item[$\bullet$ If $m=2n+1$ and $q$ is even: ] The ordered $b$-isotropic basis is
 \begin{eqnarray}\label{basis3}
e_1+ie_2,e_3+ie_4,......,e_{2n}+ie_{2n+1},e_{q+1},e_{2n}-ie_{2n+1},......,e_1-ie_2,
 \end{eqnarray}
 where $E^-:=<e_1+ie_2,e_3+ie_4,...,e_{q-1}+ie_{q},e_1-ie_2,e_3-ie_4,...,e_{q-1}-ie_{q}>$ and\\ $E^+:=<e_{q+1},e_{q+2}+ie_{q+3},...,e_{2n}+ie_{2n+1},e_{q+2}-ie_{q+3},..., e_{2n}-ie_{2n+1}>$.

\item[$\bullet$ If $m=2n+1$ and $q$ is odd: ] The ordered $b$-isotropic basis is
 \begin{eqnarray}\label{basis4}
e_1+ie_2,e_3+ie_4,......,e_{2n}+ie_{2n+1},e_q,e_{2n}-ie_{2n+1},......,e_1-ie_2,
 \end{eqnarray}
 where $E^-=<e_1+ie_2,e_3+ie_4,...,e_{q},e_1-ie_2,e_3-ie_4,...,e_{q-2}-ie_{q-1}>$ and \\$E^+=<e_{q+1}+i e_{q+2},...,e_{2n}+ie_{2n+1},e_{q+1}-ie_{q+2},..., e_{2n}-ie_{2n+1}>$.

\end{description}

\begin{rem}
If $m=2n+1$ we have the vector $e_q$ or $e_{q+1}$ in the ordered $b$-isotropic basis.
In this case the vector sits in a fixed position in the middle of the basis and the terms at the beginning skip over $e_q$ or $e_{q+1}$.
\end{rem}

\bigskip\noindent

The above bases define the standard maximal Tours $T_S$ in each case as the subgroup of diagonal matrices, i.e,
$$T_S=\{g \in  G_0:g=diag(t_1,t_2,...,t_n,-t_n,...,-t_2,-t_1), t_i \in \mathbb{C} \}~~\text{ if $m=2n$},$$
or
$$T_S=\{g \in  G_0:g=diag(t_1,t_2,...,t_n,0,-t_n,...,-t_2,-t_1), t_i \in \mathbb{C} \}~~\text{ if $m=2n+1$}.$$

 For any of the ordered bases above denote by $F_S$ the associated flag in $Z$. Reorderings on these bases will  determine the flag domains, the base cycles and the intersection points.

\bigskip\noindent

Just as in the case of $T_S$, the maximal Torus $T_I$ is defined to have a certain basis of eigenvectors which depends on $m$ being odd or even.

\begin{itemize}
	\item If $m$ is even, then the basis is   \begin{eqnarray}\label{basis 1}
 e_1+e_{2q},...,e_q+e_{q+1},e_{2q+1}+ie_{2q+2},e_{2q+3}+ie_{2q+4},...,e_{2n-1}+ie_{2n},&&  \nonumber \\
 e_{2n-1}-ie_{2n},...,e_{2q+3}-ie_{2q+4},e_{2q+1}-ie_{2q+2},e_q-e_{q+1},...,e_1-e_{2q}.& &
\end{eqnarray}
\item If $m$ is odd, then the basis is   \begin{eqnarray}\label{basis 2}
e_1+e_{2q},e_2+e_{2q-1},...,e_q+e_{q+1},e_{2q+1},e_{2q+2}+ie_{2q+3},e_{2q+4}+ie_{2q+5},...,e_{2n}+ie_{2n+1},&& \nonumber \\
e_{2q+1},e_{2n}-ie_{2n+1},...,e_{2q+4}-ie_{2q+5},e_{2q+2}-ie_{2q+3},e_q-e_{q+1},...,e_2-e_{2q-1},e_1-e_{2q}~~~&&.
\end{eqnarray}
\end{itemize}

 and $B_I$ the Iwasawa-Borel subgroup of $G$ which fixes the associated flag $F_I$.

\subsection{Weyl groups, flag domains and base cycles}
Let us use shorthand notation for the bases used above.  By $(r,s)-form$ we mean a basis $(r_1,\ldots ,r_n,s_n,\ldots s_1)$ where
$b(r_i,r_i)=b(s_i,s_i)=0$ and $b(r_i,s_i)=\delta_{ij}$ for all $i$ and $j$. This occurs in the even-dimensional case.  By $(r,t,s)$-form, which occurs in
odd-dimensional case, the $r_i$ and $s_i$ satisfy the same conditions and $t$ is a single vector with $b(t,t)=\pm 1$ and $b(t,r_i)=b(t,s_i)=0$ for
all $i$.  Here we discuss the Weyl groups $W(T_S)$ and $W(T_I)$ by their actions on these bases.  If $T$ is either of these tori, then
it stabilizes the spaces spanned by $r_i$ and $s_i$ for $i=1,\ldots ,n$ for both kinds of bases.    In this we regard $T$ as a product
$T=T_1\cdot \ldots \cdot T_n$, in the second case acting trivially on the space spanned by $t$.  In both cases an arbitrary permutation in $S_n$ acting
diagonally on these spaces by $(r_i,s_i)\mapsto (r_{\pi(i)},s_{\pi(i)})$ normalizes the $T$-action and is in the given orthogonal group $G$.
A simple reflection $(r_i,s_i)\mapsto (s_i,r_i)$ for one such $i$, and being the identity on the other $2$-dimensional spaces, is as usual denoted by a sign
change.  Such normalizes $T$, but has negative determinant.  In the first case, this means that only an even number of sign changes is allowed.
In the second case we may couple the sign change with the map $t\mapsto -t$ so that this slightly modified sign change has positive determinant.
Thus in both cases we have the action of $S_n\ltimes \mathbb{Z}_2^n$ normalizing the $T$-representation on the basis at hand.  This is in fact
the action of the full Weyl group.

\bigskip\noindent

 As in the cases of $Sp(2n,\mathbb{R})$ and $SO^*(2n)$, to describe the flag domains and their $K_0$-base cycles we use the Weyl group $W(T_S)$ of the standard torus $T_S$. The base point for $\alpha $ is the flag $F_S$ defined by the standard basis.  Then
all base points are
$F_\alpha :=w(F_S)$ where $w\in \mathbb{Z}_2^{n-1}$ if $m=2n$ and $w\in \mathbb{Z}_2^{n}$ if $m=2n+1$ is associated to the sign-change vector $\alpha $.  As before the base cycles and flag domains are the orbits $K.F_\alpha $ and $G_0.F_\alpha $, respectively, and $\alpha $ defines a non-degenerate sign with $sign(F_\alpha)=(a,b)$ and then $D_\alpha =D_{a,b}$.

\bigskip\noindent

In our particular case there are two maximal tori, the maximal torus $T_I$ and the maximal torus $T_S$ which defined above.
Recall that these two maximal tori $T_I, T_S$ are conjugate, and a conjugation induces an isomorphism $\psi$ of the associated Weyl groups.  In this case the  bijective map  $\psi$ is described, depending on the case, as follows:
\begin{description}
	\item[$\bullet$ If $m=2n$ and $q$ is even: ] Define $W_{T_S}$ to be the Weyl group with respect to the basis (\ref{basis11}) and let $W_I$ be the Weyl group with respect to the basis \ref{basis 1}.
 In this case the bijective map between $W_I$ and $W_{T_S}$ is $\psi(\mp (2i-1))=\pm i, \psi(\mp 2i)=\pm(q-i+1)$ if $1\leq i\leq  \frac{q}{2}$ and $\psi(\pm i)= \pm i $ if $i>q$.

\item[$\bullet$ If $m=2n$ and $q$ is odd: ] Define $W_{T_S}$ to be the Weyl group with respect to  basis (\ref{basis22}) and let $W_I$ be the Weyl group with respect to the basis (\ref{basis 1}), then the bijective map in this case is
  $\psi:W_I\longrightarrow W_{T_o}$ defined by  $\psi(\mp (2i-1))=\pm i, \psi(\mp 2i)=\pm(q-i)$ if $1\leq i\leq  \frac{q-1}{2}$ and $\psi(\pm i)= \pm (i-1) $ if $i>q$ and $\psi(-q)=-n. ~$

\item[$\bullet$ If $m=2n+1$ and $q$ is even: ] Define $W_{T_e}$ to be the Weyl group with respect to  the basis (\ref{basis3}) and let $W_I$ be the Weyl group with respect to the basis (\ref{basis 2}), then we can define the bijective map between them to be $\psi(\mp (2i-1))=\pm i, \psi(\mp 2i)=\pm(q-i+1)$ if $1\leq i\leq  \frac{q}{2}$ and $\psi(\pm i)= \pm i $ if $i>q$.

\item[$\bullet$ If $m=2n+1$ and $q$ is odd: ] Define $W_{T_o}$ to be the Weyl group with respect to  basis (\ref{basis4}).
Let $W_I$ be the Weyl group with respect to the basis (\ref{basis 2}). The bijective map $\psi$ between $W_I$ and $W_{T_o}$ can be defined as  $\psi(\mp (2i-1))=\pm i, \psi(\mp 2i)=\pm(q-i)$ if $1\leq i\leq  \frac{q-1}{2}$ and $\psi(\pm i)= \pm (i) $ if $i>q$.
\end{description}


\bigskip\noindent

Moreover, through an argument quiet similar to Proposition $\ref{Base cycle}$, the base cycle $C_{\alpha}$ is the set
$$C_{\alpha} =\{F \in Z : V_i= (V_i \cap E^-) \oplus ( V_i \cap E^+) , 1 \leq i \leq m\}.$$
where $E^-$ and $E^+$ defined in $\S \ref{basepoint}$ and the intersection dimensions are determined by $\alpha $ (resp. $(a,b)$). The flag domains are parametrized by the signature of its base cycles.

 \section{The real form $Sp(2p,2q)$}

\subsection{Preliminaries}
Let $G:=Sp(2n,\mathbb{C})$ be the Lie group of complex linear transformations is defined in $\S \ref{Sp}$. We realize the real form $G_0=Sp(2p,2q)$ as a group of matrices $g \in Sp(2n,\mathbb{C})$ which leave invariant the Hermitian form $h(v,w)$ of the signature $(2p,2q)$ defined by
 $$
h(v,w)=-\sum_{i=1}^{q} v_i\bar{w}_i+\sum_{i=q+1}^{p+q} v_i\bar{w}_i+\sum_{i=p+q+1}^{2p+q} v_i\bar{w}_i-\sum_{i=p+2q+1}^{2n} v_i\bar{w}_i.
$$
such that $Sp(2p,2q)=Sp(2n,\mathbb{C}) \cap SU(2p,2q)$.
A Cartan involution $ \theta:\mathfrak{sp}(2p,2q) \longrightarrow \mathfrak{sp}(2p,2q)$ in the Lie algebra level is given by $$\theta (g)=-I_{q,p,p,q} \overline{g}^t I_{q,p,p,q},~~~where~I_{q,p,p,q}=\left(\begin{array}{cccc}
                                                             -I_q & 0&0&0\\
                                                             0 & I_p&0&0\\
							                                               0 & 0&I_p&0\\
                                                             0 & 0&0&-I_q\\
                                                            \end{array}\right)$$
																														
In particular, the maximal compact subgroup of $Sp(2p,2q)$ is $K_0 := Sp(2q) \times Sp(2p)$

\subsection{Base points}
Let $Z$ be the space of all maximally $b$-isotropic full flags. By an argument which is completely analogous to that in the case of $Sp(2n,\mathbb{R})$, the standard basis $(e_1,e_2,...,e_{2n})$ of $\mathbb{C}^{2n}$ defines a base point $F_S$ in $Z$ as the associated flag
$$
<e_1>\subset <e_1,e_2>\subset ... \subset <e_1,...,e_{2n}>
$$
Reorderings on this basis will  determine the flag domains, the base cycles and the intersection points.

\bigskip\noindent

The standard basis defines the standard maximal torus $T_S$, i.e,
$$T_S:=\{g \in  G_0:g=diag(t_1,t_2,...,t_n,-t_n,...,-t_2,-t_1), t_i \in \mathbb{C} \}$$
and the standard Borel subgroup $B$ of upper triangular matrices in $Sp(2p,2q)$.

\bigskip\noindent

In order to define a base point in $\gamma^{c\ell}$ we introduce the basis
\begin{eqnarray}\label{basisA}
 e_1+e_{2q}, e_{2n-2q+1}-e_{2n},...,e_q+e_{q+1},e_{2n-q}-e_{2n-q+1},e_{2q+1},e_{2q+2},...,e_{n},e_{n+1},&&  \nonumber \\
 e_{n+2},...,e_{2n-q},e_q-e_{q+1},e_{2n-q}+e_{2n-q+1},...,e_1-e_{2q},e_{2n-2q+1}+e_{2n}.~~~~~~~~~& &
\end{eqnarray}
with $T_I$ being the maximal torus which fixes the basis in the sense of eigenvectors.
\subsection{Weyl group, flag domains and base cycles}
By arguments which are completely analogous to those in $\S \ref{Weyl}$, if
$T$ is a maximal torus defined by a basis
of $(r,s)$-type, then $W(T)=S_n\ltimes \mathbb{Z}_2^{n}$

 As in the case of $Sp(2n,\mathbb{R})$, to describe the flag domains and their $K_0$-base cycles we use the Weyl group $W(T_S)$ of the standard torus $T_S$. The base point for $\alpha $ is the flag $F_S$ defined by the standard basis.  Then
all base points are
$F_\alpha :=w(F_S)$ where $w\in \mathbb{Z}_2^{n}$ is associated to the sign-change vector $\alpha $.  As before the base cycles and flag domains are the orbits $K.F_\alpha $ and $G_0.F_\alpha $, respectively, and $\alpha $ defines a non-degenerate sign with $sign(F_\alpha)=(a,b)$ and then $D_\alpha =D_{a,b}$.

\bigskip\noindent

The maximal tori $T_I$ and $T_S$ defined in the previous section are conjugate, and a conjugation induces an isomorphism $\psi$ of the associated Weyl groups.  In this case the bijective $\psi( \pm(2i-1))=\pm (2q-i+1),~ \psi(\pm 2i)=\mp i$ if $1\leq i\leq  q$ and $\psi(\pm i)= \pm i $ if $i>2q$. Note that $-(2q-i+1)=2n-2q-i+2$.

\bigskip\noindent

Finally, Define $E^-:=<e_1,...,e_{q},e_{2n-q+1},...,e_{2n}>$ and $E^+:=<e_{q+1},...,e_{2n-q}>$.
By the same argument as $\S \ref{flag}$ and Proposition \ref{Base cycle}, for a fixed flag domain $D_{\alpha}$, the base cycle $C_{\alpha}$ is the set
$$C_{\alpha} =\{F \in Z : V_i=(V_i \cap E^-) \oplus ( V_i \cap E^+), 1 \leq i \leq 2n\}.$$
As all cases above the intersection dimensions are determined by $\alpha $ (resp. $(a,b)$).


\chapter{Fixed Point Theorem}
\section{Introduction}
Here, as throughout our work, we consider a homogeneous $G$-manifold of full flags $Z$ equipped with the actions of a real form $G_0$ as  listed in the previous chapter.  The Iwasawa-Borel subgroup $B_I$ is the isotropy group at the flag $F_I$
given there by an explicit basis $\mathcal{B}_I$ which also defines the maximal torus $T_I$ and its Weyl group $W_I$. For $w\in W_I$ we study the intersections
$S_w\cap C_\alpha $ of the Schubert cells $S_w=B_I.F_w$ with any base cycle $C_\alpha $, where $F_w=w(F_I)$ .

\bigskip\noindent

On the other hand, we have a standard basis $\mathcal{B}_S$ which defines a torus $T_S$ and a base point $F_S$ in a standard base cycle which when moved by sign permutations in the Weyl group $W_S$ gives base points in all other base cycles $C_{\alpha}$. Recall that these base cycles are the closed $K$-orbits which are described by the intersections with $E^+$ and $E^-$ in each case. All of the above has been defined in Chapter $2$ for all cases.

\bigskip\noindent

In this chapter we first give detailed formulas for the orbits $b(F_I), b \in B_I$, case by case for all real forms. Then we prove a "Moving Lemma"(See the Moving Lemma \ref{Move}): Given a point $b.w(F_I)$ in $C_\alpha $, we move it continously in $C_\alpha$ to a $T_S$-fixed point by starting with the basis given for $b.w(F_I)$ and moving continuously by a curve $\gamma (t)$ of bases $\mathcal{B}(t)$ which defines flags in $S_w\cap C_\alpha $. In general the intersection $S_w \cap C_{\alpha}$ may have many connected components. The moving procedure implies that in each component of $S_w\cap C_\alpha $ there
is a fixed point. In the case where $S_w$ is of complementary dimension to $C_{\alpha}$ the intersection $S_w \cap C_{\alpha}$ is finite; so the curve $\gamma (t)$ is just a fixed point and the result states that $S_w \cap C_{\alpha}$ consists only of $T_S$-fixed points (See Corollary \ref{fixthm}).

\bigskip\noindent


\section{Description of the $B_I$-orbits}\label{S3.2}
In the previous chapter in every case we have chosen a flag $F_I$ defined by the basis $(v_1,\ldots ,v_{m}),m=2n$ or $m=2n+1$, which belongs to the $G_0$-closed  orbit $\gamma^{cl}$. Given a Weyl element $w\in W$, we consider the Schubert cell $B_I.F_w$ and an element $F=b(F_w)$  in it which is defined by an ordered basis $(b.v_{w_1},\ldots ,b.v_{w_{m}})$, $b \in B_I$.
If $F\in C_{\alpha}$, then it can be defined by an ordered basis $(\varepsilon_1,\ldots ,\varepsilon_{m})$
where for all $k$ the vector $\varepsilon_k$ is either in $E^-$ or $E^+$ where $V=E^+ \oplus E^-$, and is a linear combination
$$
\varepsilon_k=\sum_{j\le k} c_jb.v_{w_j}\,.
$$
In this case we call the basis a \textit{\textbf{split basis}}.
As a result of the discussion in the previous section we can write formulas for  $b.v_{w_j}$ for any element $b \in B_I$ for each real form as follows:
\begin{description}
\item[1. If $G_0=SP(2n,\mathbb{R})$,] then
	$$
b.v_{w_j}=\eta_j(e_r +e_{2n-r+1})+\zeta_j(e_r-e_{2n-r+1})+B_j=K_j+B_j,~1\leq r \leq n
$$
where  $B_j$ does not involve the basis vectors $e_r $ and $e_{2n-r+1}$.
\item[2. If $G_0=SO^*(2n)$,] then
	$$
b.v_{w_j}=\eta_j(e_r +ie_{n+r})+\zeta_j(e_r-ie_{n+r})+B_j=K_j+B_j,~1\leq r \leq n
$$
where  $B_j$ does not involve the basis vectors $e_r $ and $e_{n+r}$.
	\item[3. If $G_0=SO(p,q)$,] then
	\begin{description}
		\item[(i)] $b.v_{w_j}=\eta_{ij}(e_r +e_{2q-r+1})+\zeta_{ij}(e_r-e_{2q-r+1})+\eta_{kj}(e_{r+1} +e_{2q-r})+\zeta_{kj}(e_{r+1}-e_{2q-r})  +B_j\\=K_j+B_j, $ where $~1\leq r \leq q$
and $\eta_{kj}= \pm i \eta_{ij}$ and $\zeta_{kj}= \pm i \zeta_{ij}$ and $B_j$ does not involve the basis vectors $e_r $, $e_{r+1}$, $e_{2q-r+1}$ and $e_{2q-r}$.
\item[(ii)] $b.v_{w_j}=\eta_j(e_{m-2r+1} +ie_{m-2r+2})+B_j=K_j+B_j,~1\leq r \leq \frac{m-2q}{2}$  \\
where $B_j$ does not involve the basis vectors $e_{m-2r+1}$ and $e_{m-2r+2}$.
	\item[(iii)] $b.v_{w_j}=\eta_j e_{2q+1}+B_j=K_j+B_j,$  \\
where $B_j$ does not involve the basis vectors $e_{2q+1}$.
	
	\end{description}
	
\item[4. If $G_0=SP(2p,2q)$,] then
\begin{description}
	\item[(i)] $b.v_{w_j}=\eta_j(e_r +e_{2q-r+1})+\zeta_j(e_r-e_{2q-r+1})+B_j=K_j+B_j, $ where $~1\leq r \leq q$
where  $B_j$ does not involve the basis vectors $e_r $ and $e_{2q-r+1}$.

\item[(ii)] $b.v_{w_j}=\eta_j(e_{2n-q+r} + e_{2n-r+1})+\zeta_j(e_{2n-q+r} - e_{2n-r+1})+B_j=\tilde{K}_j+\tilde{B}_j,$ where $~1\leq r \leq q$
where  $B_j$ does not involve the basis vectors $e_{2n-q+r} $ and $e_{2n-r+1}$.

\item[(iii)] $b.v_{w_j}=\eta_j(e_{2q+r})+B_j.$ where $~1\leq r \leq 2p$
where  $B_j$ does not involve the basis vector $e_{2q+r}$.

\end{description}
\end{description}

\section{The Moving Lemma and a Fixed Point Theorem}

 Assume that the Schubert cell $B_I.F_w$ has non-empty intersection with $D_{\alpha}$. This implies that
it has a non-empty intersection with the base cycle $C_{\alpha}$. The flags $F_0$ in this cycle
can be described by bases $\mathfrak{B}(F_0)$ which are split in the sense that their entries
(vectors) are either in $E^-$ or $E^+$ where $V=E^- \oplus E^+$.

\bigskip\noindent

\begin{lem}(\text{\textbf{The Moving Lemma}}).\label{Move}
Given a flag $F_0 \in B_I.F_I \cap C_{\alpha}$, there is a continuous curve of split bases $\mathcal{F}=\mathfrak{B}(t), t \in [0,t_0]$ such that $\mathfrak{B}(t_0)$ defines the given flag $F_0$ and $\mathcal{B}(0)$ defines a $T_S$-fixed point.
	\end{lem}

\begin{proof}
The fixed point $F_I$ in $\gamma^{cl}$ is defined by the ordered basis
$$(v_1,v_2,...,v_m)$$
 and $F_w=w(F_I)$. Let $b.F_w=:F_0, b \in B_I$, be a given flag in $B_I.F_w \cap C_{\alpha}$ and regard it as being given
by a basis $$(b.v_{w_1},\ldots ,b.v_{w_m})$$ which is given in $\S \ref{S3.2}$. Of course this basis is not necessarily split. However, the points which belongs to $C_\alpha $ are given by split bases, and we will now construct an associated split basis $(\varepsilon_1,\ldots ,\varepsilon_{m})$ which belong to $S_w \cap C_{\alpha}$ (term-by-term, recursively). First, we analyze $\varepsilon_1$ and describe it in a way that is convenient for our purposes. There are three possibilities for
$$\varepsilon_1 =b.v_{w_1}=\lambda v_{w_1}+ B_{w_1}, $$ where $B_{w_1}=\sum_{j\le w_1} c_jv_{j},$ i.e.,
\begin{enumerate}
	\item If $1\leq {w_1} \leq q$  \footnote{$q$ comes from the signature $(p.q)$ of the bilinear form $b$ where $q=n$ in the cases $Sp(2n,\mathbb{R})$ and $SO^*(2n)$.} such that the dual vector \footnote{If $v_i,\tilde{v}_i \in \mathbb{C}^m$ and $b(v_i,v_i)=0$,$b(\tilde{v}_i,\tilde{v}_i)=0$ and $b(\tilde{v}_i,v_i)=1$ then $\tilde{v}_i$ is the dual vector of $v_i$.} $v_{m-w_1+1}$ is not involved as a term of $B_{w_1}$, so there is no possibility of canceling the term $\lambda v_{w_1}$ with $B_{w_1}$.
	
	\item  If $ -1 \leq {w_1} \leq -q$ such that the dual vector $v_{m-w_1+1}$ is  involved as a term in $B_{w_1}$, so  we can find $2 \lambda e_j$ as combination of these two vectors.
	
	\item If $q+1 \leq |w_1| \leq \left\lfloor \frac{m}{2} \right\rfloor$, this case appear for the real forms $G_0=SP(p,q)$ and $G_0=SO(p,q)$. In this case $v_{w_1}=e_r$ for some $r$ if $G_0=SP(p,q)$, and $v_{w_1}=e_r \pm i e_{r+1}$ for some $r$  or $v_{w_1}=e_q$ or $v_{w_1}=e_{q+1}$  if $G_0=SO(p,q)$.
	\end{enumerate}
Note that such combinations in those cases are necessary for the first vector in the basis $\mathcal{B}(F_0)$. Now a number of the other vectors $v_i$ occur in such pairs and must be combined in the same way. However, the resulting coefficient is an arbitrary complex number.
As a result, if $\tilde{\lambda}$ is the coefficient of the dual vector $v_{m-w_1+1}$, then the first basis element is $$ \varepsilon_1= 2\lambda e_j + t_1e_{j+1}+...+ t_{n-j}e_n, \text{~~if~} \tilde{\lambda}=\lambda$$ or
$$\varepsilon_1 = 2\lambda e_{m-j+1} + t_1e_{m-j}+...+ t_{m-j}e_{n+1}  \text{~~if~} \tilde{\lambda}=-\lambda.$$
where $m=2n$ or $m=2n+1$. We can not know the values of the coefficients, but only know that the $t_i$ can in principle be any complex number and $\lambda \neq 0$.  Note that the first case $\varepsilon_1=b.v_{w_1}$ where $ 1\leq w_1 \leq q$ doesn't happen, because there is no possible linear combination which produces a split basis element.\\
Next we construct and describe $\varepsilon_2$. Recall that the vector $b.v_{w_2}$ is given by the form
$$b.v_{w_2}=\lambda v_{w_2}+B_{w_2}, $$
where $B_{w_2}=\sum_{j\le w_2} c_jv_{j}$, but we can modify it by adding in an appropriate multiple of $\varepsilon_1$, i.e.
$$\varepsilon_2 =\lambda v_{w_2}+d_1 \varepsilon_1+B_{w_2}. $$  If we use $d_1 \varepsilon_1$ in the linear combination which defined $\varepsilon_2$, we will call this combination "external combination", if not, we call this combination "internal combination". Then there are four possibilities for $\varepsilon_2$ i.e.,
\begin{description}
	\item[ A. ] If $1\leq {w_2} \leq q$ such that the dual vector $v_{m-w_2+1}$ is not involved as a term of $B_{w_2}$ and $ \varepsilon_1 \neq v_{m-w_2+1}$ , so there is no possibility of canceling the term $\lambda v_{w_2}$ with $B_{w_2}$ or with $\varepsilon_1$.
	
	\item[ B. ] If $1\leq {w_2} \leq q$ such that the dual vector $v_{m-w_2+1}$ is not involved as a term of $B_{w_2}$ but $ \varepsilon_1 = v_{m-w_2+1}$ , so  we can find $2 \lambda e_j$ as combination of $\varepsilon_1$ and $v_{w_2}$.
	
	\item[ C. ]  If $ -1 \leq {w_2} \leq -q$ such that the dual vector $v_{m-w_1+1}$ is  involved as a term in $B_{w_2}$, so  we can find $2 \lambda e_j$ as combination of these two vectors.
	
	\item[ D. ] If $q+1 \leq |w_2| \leq \left\lfloor  \frac{m}{2} \right\rfloor$, this case appear for the real forms $G_0=SP(p,q)$ and $G_0=SO(p,q)$. In this case $v_{w_2}=e_r$ for some $r$ if $G_0=SP(p,q)$, and $v_{w_2}=e_r \pm i e_{r+1}$ for some $r$  or $v_{w_2}=e_q$ or $v_{w_2}=e_{q+1}$  if $G_0=SO(p,q)$.
	\end{description}
	The combinations in those cases are necessary for the second vector in the basis $\mathcal{B}(F_0)$ and all other vectors $\varepsilon_j, j\geq 2$ . Now a number of the other vectors $v_i$ occur in such pairs and must be combined in the same way.
As a result, if $\tilde{\lambda}$ is the coefficient of the dual vector $v_{m-w_2+1}$, then the second basis element is $$ \varepsilon_2= 2\lambda e_j + t_1e_{j+1}+...+ t_{n-j}e_n, \text{ if } \tilde{\lambda}=\lambda, $$
$$\varepsilon_2 = 2\lambda e_{m-j+1} + t_1e_{m-j}+...+ t_{m-j}e_{n+1}, \text{ if }  \tilde{\lambda}=-\lambda .$$
where $m=2n$ or $m=2n+1$, $t_i$ is complex number and $\lambda \neq 0$.

We now assume that $\varepsilon_1,\ldots ,\varepsilon_{j-1}$ have been constructed as in the case of $\varepsilon_2$ depending on $\tilde{\lambda}$. This means that $\varepsilon_\ell =2\lambda e_j + t_1e_{j+1}+...+ t_{n-j}e_n, \text{ if } \tilde{\lambda}=\lambda,$ or $\varepsilon_\ell =2\lambda e_{m-j+1} + t_1e_{m-j}+...+ t_{m-j}e_{n+1}, \text{ if }  \tilde{\lambda}=-\lambda .$ for $\ell \le j-1$. Now we will construct $\varepsilon_j$ with these properties. Recall that the vector $$b.v_{w_j}=\lambda v_{w_j}+B_{w_j}, $$ where $B_{w_j}=\sum_{j\le w_2} c_jv_{j},$ but we can modify it by adding in an appropriate linear combination of the $\varepsilon_i,1\leq i \leq j-1$, i.e. $$\varepsilon_j =\lambda v_{w_j}+B_{w_j}+ \sum_{i<j}d_i \varepsilon_i, $$ then there are only three cases for $\varepsilon_j$ :
\begin{enumerate}
	\item Some of the coefficients of $B_j$ are such that cancellation takes
place as step $C$ above which called internal cancellation.
\item There is a linear combination of the previous split basis elements $\sum_{i<j}d_i \varepsilon_i$ that
cancels one of the terms in expressions of the type $e_s \pm e_r$. Similarly as step $B$ above which called external cancellation..
\item The vector $v_j$ is equal to $e_r$ in the case of the real form $SP(p,q)$ or $v_j$ is equal to $e_r \pm i e_{r+1}$ or $e_r$ in the case of the real form $SO(p,q)$. Similarly as step $D$ above which is a kind of internal cancellation..
\end{enumerate}
As a result, if $\tilde{\lambda}$ is the coefficient of the dual vector $v_{m-w_1+1}$, then the $j$-th basis element is $$ \varepsilon_j= 2\lambda e_s + t_1e_{s+1}+...+ t_{n-s}e_n, \tilde{\lambda}= \lambda,$$ or
$$\varepsilon_j = 2\lambda e_{m-s+1} + t_1e_{m-s}+...+ t_{m-s}e_{n+1}, \tilde{\lambda}= - \lambda,$$
 where $m=2n$ or $m=2n+1$ and $t_i$ is complex number and $\lambda \neq 0$.

\bigskip\noindent

 The next claim is we can move continuously from an arbitrary $\varepsilon$-basis constructed above to a $T_S$-fixed point continuously, i.e., without leaving $C_{\alpha}$.
The curve
$$\mathfrak{B}(t)=(\varepsilon_1(t),\ldots ,\varepsilon_m(t)),$$ which begins at the
given basis $(\varepsilon_1,\ldots ,\varepsilon_m)$ and ends at a $T_S$-fixed
bases $(\varepsilon_1(0),\ldots ,\varepsilon_m(0))$, will be constructed stepwise.
First we show that there is a curve $\varepsilon_1(t)$ beginning at the given vector
$\varepsilon_1$ and ending at a $T_S$-fixed vector $\varepsilon_1(0)$.  With this
we reach a new base point $$(\varepsilon_1(0),\varepsilon_2,\ldots ,\varepsilon_m).$$
If the vector $\varepsilon_1$ in the given base point  is $T_S$-fixed point, then  the curve $\varepsilon_1(t)=\varepsilon_1(0)$ for all $t \in[0,t_0]$, i.e a $T_S$-fixed point.
Then we assume that the given base point $(\varepsilon_1,\ldots ,\varepsilon_m)$ has
the property that $\varepsilon_\ell $ is $T_S$-fixed for $\ell <k $ such that the curve $\varepsilon_\ell(t)=\varepsilon_\ell(0), \ell<k$ for all $t \in[0,t_0]$ and then construct
a curve $\varepsilon_k(t)$ beginning at $\varepsilon_k$ and ending at  a $T_S$-fixed
vector $\varepsilon_k(0)$. This defines a new base point in $S_w\cap C_\alpha $ with $\varepsilon_\ell$
begin $T_S$-fixed for $\ell \le k$.  As a result we define a sequence of curves whose composition
is the desired continuous curve.


Suppose $F \in S_w \cap C_{\alpha}$ is associated to a split basis and  not $T_S$-fixed. We may therefore choose an $\varepsilon$-basis with a smallest index $k$ such that $\varepsilon_k$
is not fixed and $\varepsilon_\ell $ is fixed for $\ell < k$.

So we can define the curve such that  $\varepsilon_{\ell}$ is $T_S$-fixed ,  then we will construct  the curve $\varepsilon_k(t)$ where $\varepsilon_k$ is the first non $T_S$-fixed.  

 It is sufficient to consider the case where $\varepsilon_k\in V^-$.
The first step is to analyze the $\varepsilon_k$, but they are constructed in such a way that for the movement $b(t)$  we have
$$\varepsilon_k=(\lambda_1 e_r +\sum_{j\le q} a_je_j)+i(\lambda_2 e_{r+1} +\sum_{s\le q} b_se_s)$$ where $a_r=0$,$b_{r+1}=0$ and $a_j\not=0$, $b_s\not=0$ for some other $j$ and $s$.  Furthermore, if $\varepsilon_\ell =e_j $
for some $\ell < r $,   then $a_j=b_s=\lambda_2=0$, and if $\varepsilon_\ell =e_j+ie_{j+1}$
for some $\ell < r $,   then $a_j=b_s=0$. In general,  if $G_0=SO(p,q)$, then $\lambda_1=\pm \lambda_2$, and if $G_0\neq SO(p,q)$, then $\lambda_2=b_s=0$.

If $k=1$ then the curve is $$
\varepsilon_1(t)=c_1 b(t).v_{w_1}\,.
$$
It follows that
$$
\varepsilon_1 (t)=\varepsilon_1 +(t-1) K_1\,.
$$
where  $K_1=0$~ corresponds to~ $e_r+ie_{r+1}$~ or ~$e_r$~ not being involved in $\varepsilon_1$.  Note that $\varepsilon_1 \in V^-$ implies that $\zeta_1=0$ and $\tilde{\zeta}_1=0$, and $\varepsilon_1 \in V^+$ implies that $\eta_1=0$ and $\tilde{\eta}_1=0$ in the formulas of section \ref{S3.2}.   Thus $\varepsilon_1 (t)\in V^-$ (resp. $V^+$)
 whenever $\varepsilon_1 \in V^-$(resp. $V^+$).  In other words, the image of the curve $\gamma $ is contained in $S_w\cap C_\alpha $.
 Finally, observe that $\eta_1 $ is the coefficient of $e_r $ in $\varepsilon_1$.  This has be scaled to $1$ and as a result the coefficient
 of $e_r $ or $e_r+ie_{r+1}$ in $\varepsilon_1(t)$ is $t$.  Thus
$$\varepsilon_1(t)=\varepsilon_1+(t-1)e_r, \text{~if~} G_0 \neq SO(p,q),$$ or
$$\varepsilon_1(t)=\varepsilon_1+(t-1)(e_r+ie_{r+1}), \text{~if~} G_0 = SO(p,q).$$
  In the various cases we then have the following descriptions of $\varepsilon_1(0)$;
\begin{description}
	\item[$\bullet$ If $G_0=SP(2n,\mathbb{R})$,] then
	$$
	\varepsilon_1(0)=\varepsilon_1-e_r=[(e_r +e_{2n-r+1})]-e_r=e_{2n-r+1},~1\leq r \leq n.
$$

\item[$\bullet$ If $G_0=SO^*(2n)$,] then
	$$
\varepsilon_1(0)=\varepsilon_1-e_r=[(e_r +ie_{n+r})]-e_r=ie_{n+r},~1\leq r \leq n,
$$or
$$
\varepsilon_1(0)=\varepsilon_1-e_r=[-i(e_{r-n} +e_{r})]-e_r=-ie_{r-n},~n+1\leq r \leq 2n.
$$
	\item[$\bullet$ If $G_0=SO(p,q)$,] then
		$$\varepsilon_1(0)=\varepsilon_1-(e_r \mp ie_{r+1})=(e_r  \mp ie_{r+1})+(e_{2q-r} \mp ie_{2q-r+1})-(e_r \mp ie_{r+1})=(e_{2q-r} \mp ie_{2q-r+1}),$$
	where	$~1\leq r \leq q.$
		\item[$\bullet$ If $G_0=SP(2p,2q)$,] then
	 $$\varepsilon_1(0)=\varepsilon_1-e_r=(e_r +e_{2q-r+1})-e_r=e_{2q-r+1}, ~1\leq r \leq q.$$
\end{description}

\bigskip\noindent
Recall that the above non fixed $\varepsilon_k$ is $\varepsilon_k=b.v_{w_k}=K_k+B_k$ which defined for all cases in $\S \ref{S3.2}$ where $B_k$ does not involve the terms of $K_k$. In the following we will define a curve $\gamma$ which sent the non fixed point to $T_S$-fixed point . Define $b(t)\in B_I$ by $b(t)(v_{w_j})=tK_j+B_j$. The curve $\gamma $ is defined by $t\mapsto (b(t)(v_{w_1}),\ldots , b(t)(v_{w_{2n}}))$ which is defined by the corresponding $\varepsilon $-basis,
$$
\varepsilon_k(t)=\sum_{j\le k} c_jb(t).v_{w_j}\,.
$$
It follows that
$$
\varepsilon_\kappa (t)=\varepsilon_\kappa +(t-1)\sum_{j\le \kappa} c_{\kappa j}K_j\,.
$$
Notice that $\sum_{j\le \kappa} c_{\kappa j}K_j=0$~ corresponds to~ $e_r+ie_{r+1}$~ or ~$e_r$~ not being involved in $\varepsilon_\kappa$.  Therefore $\varepsilon_\kappa (t)=\varepsilon _\kappa $ for $\kappa <k$.   In general, $\varepsilon_\kappa \in V^-$ implies that
$\sum c_{\kappa j}\zeta_j=0$ and $\sum c_{\kappa j}\tilde{\zeta}_j=0$, and $\varepsilon_\kappa \in V^+$ implies that $\sum c_{\kappa j}\eta_j=0$ and $\sum c_{\kappa j}\tilde{\eta}_j=0$ in the formulas of section \ref{S3.2}.   Thus $\varepsilon_\kappa (t)\in V^-$ (resp. $V^+$)
 whenever $\varepsilon_\kappa \in V^-$(resp. $V^+$).  In other words, the image of the curve $\gamma $ is contained in $S_w\cap C_\alpha $.
 Finally, observe that $\sum c_{\kappa j}\eta_j $ is the coefficient of $e_r $ in $\varepsilon_k$.  This has be scaled to $1$ and as a result the coefficient
 of $e_r $ or $e_r+ie_{r+1}$ in $\varepsilon_k(t)$ is $t$.  Thus
$$\varepsilon_k(t)=\varepsilon_k+(t-1)e_r, \text{~if~} G_0 \neq SO(p,q),$$ or
$$\varepsilon_k(t)=\varepsilon_k+(t-1)(e_r+ie_{r+1}), \text{~if~} G_0 = SO(p,q).$$
 Since at least one $a_j\not=0$, and one $b_j\not=0$
 and $\varepsilon_\ell (t)=\varepsilon_\ell $ for $\ell <k$. In the various cases we then have the following descriptions of $\varepsilon_k(0)$;
\begin{description}
	\item[$\bullet$ If $G_0=SP(2n,\mathbb{R})$,] then
	$$
	\varepsilon_k(0)=\varepsilon_k-e_r=[(e_r +e_{2n-r+1})]-e_r=e_{2n-r+1},~1\leq r \leq n.
$$

\item[$\bullet$ If $G_0=SO^*(2n)$,] then
	$$
\varepsilon_k(0)=\varepsilon_k-e_r=[(e_r +ie_{n+r})]-e_r=ie_{n+r},~1\leq r \leq n,
$$or
$$
\varepsilon_k(0)=\varepsilon_k-e_r=[-i(e_{r-n} +e_{r})]-e_r=-ie_{r-n},~n+1\leq r \leq 2n.
$$
	\item[$\bullet$ If $G_0=SO(p,q)$,] then
		$$\varepsilon_k(0)=\varepsilon_k-(e_r \mp ie_{r+1})=(e_r  \mp ie_{r+1})+(e_{2q-r} \mp ie_{2q-r+1})-(e_r \mp ie_{r+1})=(e_{2q-r} \mp ie_{2q-r+1}),$$ $ ~1\leq r \leq q.$
		\item[$\bullet$ If $G_0=SP(2p,2q)$,] then
	 $$\varepsilon_k(0)=\varepsilon_k-e_r=(e_r +e_{2q-r+1})-e_r=e_{2q-r+1}, ~1\leq r \leq q.$$
\end{description}
all of the above give us a vector in the standard basis $T_S$ which implies that we move the point to a $T_S$-fixed point. The above recursive construction defines the $\varepsilon(t)_j$ as desired and these in turn define the curve $\mathfrak{B}(t)$ of bases.
\end{proof}

The above proof shows that if $v(t)$ is any of the basis vectors $\varepsilon_j(t)$ in $\mathfrak{B}(t)$,
then the following hold.

\begin{rems}\label{rem1}

\begin{enumerate}

	\item  At most one coefficient of $v(t)$ in $\mathfrak{B}(t)$ is restricted not to vanish (This happens as in the case $C$ where $2 \lambda$ arises as a coefficient!).
  \item Except for the previous case, all coefficients are not restricted. However, they may or may not be related, e.g., in the case where an external linear combination was used.
   \item In all cases there are limiting vectors when various $t_i$ go to $0$ which are single elements of the standard basis $(e_1,..., e_{2n})$ except for the real form $SO(p,q)$ the basis is
	
	$(e_1 \pm ie_2,....,e_{m-1} \pm ie_{m})$.
\end{enumerate}

\end{rems}

A Schubert cell intersects $C_\alpha $ in several connected component (sometimes only one as the case $Sp(2n,\mathbb{R})$, but sometimes many as the case $SO(p,q)$). Since the curve given by the Moving Lemma is continuous, it follows that every such component contains a $T_S$-fixed point.

\begin{thm}(\text{\textbf{The Fixed Point Theorem}}).\label{corcor}
Every component of the intersection $S_w \cap C_{\alpha}$ has a fixed point.\\
\end{thm}

\begin{cor}\label{fixthm}
If $codim(S_w)=dim(C_\alpha)$ then
$$
S_w\cap C_\alpha \subset Fix(T_s)\,.
$$
\end {cor}
\begin{proof}
If $S_w$ of complementary dimension, then the components of intersection are isolated points. By using Theorem \ref{corcor} every intersection point is a fixed point.\\
\end{proof}


\chapter{Cycle intersection for $SP(2n,\mathbb{R})$}

 \section{Conditions for $S_w \cap C_{\alpha} \neq \emptyset$}
In this chapter the manifold $Z$ under consideration is the homogeneous space of maximally isotropic full flags of  the complex symplectic group $G=Sp(2n,\mathbb{C})$ equipped  with the action of the real form $G_0=Sp(2n,\mathbb{R})$.
Recall that $F_I$ is the fixed point in the closed orbit $\gamma^{cl}$ with isotropy group $B_I$, i.e., which is associated to the basis (\ref{basisAA}).

\bigskip\noindent

As described in $\S 2.2.3$ the Weyl group $W_I$ associated to $T_I$ acts on the basis (\ref{basisAA}) by permutation plus sign change. If $w\in W_I$ is regarded as an element of $G$, then $F_w:=w(F_I)$ and $S_w:=B_I.F_w$ is the associated Iwasawa-Schubert cell. The dimension of the cells corresponds to the length of the word $w$ in the Weyl group, i.e. if $F_w=w.F_I$, then $dim(B.F_w)=l(w)$. \\
Below we will describe the Weyl elements which parametrize the Iwasawa-Schubert cells which have non-empty intersections with flag domains (See Theorem \ref{thmSP}). For this we call attention to a particular class of Weyl elements.

\bigskip\noindent

\begin{defn}
An element $w=(w_1,w_2,...,w_n) \in S_n \ltimes \mathbb{Z}_2^{n}$ is called \textit{\textbf{a generous permutation}}\footnote{Recall that we define a special way of denoting to write the Weyl elements in section $\S 2.2.3$.} if $w_i<0$ for all $1\leq i\leq n$.
\end{defn}
 \begin{rem}\label{rem333}
When discussing the $B_I$-orbit of the base point $F_I$ it is important to explicitly understand the orbits $B_I.(e_i-e_{2n-i+1})$ and $B_I.(e_i+e_{2n-i+1})$. As we explained in $\S 3.2$ we have the following facts:\\
If $F_I$ and $B_I$ are as above, then the orbits of interest are $B_I.(e_i\pm e_{2n-i+1})$, $i=1,\ldots ,n$. In this case the orbits $B_I.(e_i+ e_{2n-i+1})$ and $B_I.(e_i - e_{2n-i+1})$ have points of the forms
$$b(e_i+e_{2n-i+1})=\lambda(e_i+e_{2n-i+1})+\ldots  +b_n(e_{n}+e_{n+1})+a_n(e_n-e_{n+1})+\ldots +a_1(e_{1}-e_{2n})$$
and
$$b(e_i-e_{2n-i+1})=\lambda(e_i-e_{2n-i+1})+a_{i-1}(e_{i-1}-e_{2n-i})+\ldots +a_1(e_{1}-e_{2n})$$
 respectively, with $\lambda\not=0$. Note that in the above orbits, if $b$ is chosen appropriately, then we can arrange $b(e_i+e_{2n+i-1})=e_i$ or $b(e_i+e_{2n+i-1})=e_{2n-i+1}$ for all $i$. This plays a role in the proof of Theorem \ref{thmSP}.\qed
\end{rem}
\bigskip\noindent
Recall that our base point in the closed orbit $\gamma^{cl}$ is the flag $F_I$ associated to the ordered basis
\begin{equation}\label{basisI}
e_1-e_{2n},e_2-e_{2n-1}...,e_n-e_{n+1},e_n+e_{n+1},...,e_2+e_{2n-1},e_1+e_{2n}
\end{equation}
and the base point which defined the flag domains and the base cycles is the flag associated to the ordered basis
\begin{equation}\label{basisC}
e_1,e_2,e_3,......,e_{2n-1},e_{2n}
\end{equation}
These two bases defined two maximal tori which are conjugate and the corresponding Weyl groups $W_I$ and $W_{T_S}$ are isomorphic. \\
Let $\psi$ be  a bijective map between  Weyl groups $W_I$ and $W_{T_S}$ defined by $\psi(\mp (i))=\pm i,$ for all $1\leq i\leq n$, then we have the following Proposition;

\bigskip\noindent


\begin{prop}
If $w$ is a generous permutation, then the flag $F_{\psi(w)}$ belongs to the orbit $B_I(F_w)$.
\end{prop}
\begin{proof}
Let $w \in W_I$ be a generous permutation, and $F_w=w(F_I) \in \gamma^{cl}$ be the isotropic full flag associated to $w$ such that the first $n$-subspaces in the flag are
$$\{0\} \subset <v_{w_1}> \subset <v_{w_1},v_{w_2}> \subset ...... \subset <v_{w_1},...,v_{w_n}>$$
where $v_{w_i}=e_{j}\pm e_{2n-j+1}$. We will show that the orbit $B_I(F_w)$ contains the flag $F_{\psi(w)}$.
 For this purpose let $\tilde{w} \in W_{S} $ be the image of $w$ under the bijective map $\psi$ and let $Y_{\tilde{w}}$ be the isotropic flag associated to $\tilde{w}$, i.e., the flag $$Y_{\tilde{w}}:=(\{0\} \subset <e_{\tilde{w}_1}> \subset <e_{\tilde{w}_1},e_{\tilde{w}_2}> \subset ...... \subset <e_{\tilde{w}_1},...,e_{\tilde{w}_n}>)$$
We will show that $Y_{\tilde w} =b(F_w)$ for $b$ appropriately chosen in $B_I$. For this recall that since $w$ is a generous permutation it follows that $w_i<0$ for all $1\leq i\leq n$
, and therefore $B_I.(e_j+e_{2n-j+1})$, $j=1,\ldots ,n$, are the orbits of interest. By Remark \ref{rem333} the orbit $B_I.<e_{i}+e_{2n-i+1}>$ contains the points
$$y=b(e_i+e_{2n-i+1})=\lambda(e_i+e_{2n-i+1})+\ldots  +b_n(e_{n}+e_{n+1})+a_n(e_n-e_{n+1})+\ldots +a_1(e_{1}-e_{2n})$$
where $\lambda\neq 0$. Choose all constants to be $0$ except $\lambda=a_i=\frac{1}{2}$ , then $y=<e_i>=<e_{\tilde{w}(i)}>$.\\
We can repeat this step n-times to have the first half of the flag $Y_{\tilde{w}}$, then we can extend this flag to maximal $h$-isotropic flag, . Therefore, the flag $Y_{\tilde{w}}$ is constructed.
\end{proof}
\begin{thm}\label{thmSP}\textbf{(Generous Permutation Theorem)}.
The following are equivalent
\begin{description}
	\item[~~~~~~~(i)]  $w$ is generous.
  \item[~~~~~~~(ii)]  $B_I(F_w)\cap D_\alpha \not=\emptyset$ for some $\alpha $.
  \item[~~~~~~~(iii)] $B_I(F_w)\cap D_\alpha \not=\emptyset$ for every $\alpha $.
		\end{description}
Furthermore, under any of these conditions, for every $\alpha $ the intersection $B_I(F_w)\cap C_\alpha$
 contains a $T_S$ fixed point.
\end{thm}
\begin{proof}
\begin{description}
	\item[(iii)$\Rightarrow$ (ii)] Trivial direction.
	\item[(ii)$\Rightarrow$ (i)] We assume by contradiction that $w$ is not a generous permutation and will show that $S_w \cap D_{\alpha} = \emptyset,$ for all $\alpha$, i.e. $S_w$ has no $T_S$-fixed points\footnote{See Corollary \ref{fixthm}}. Recall that the complex bilinear form $b$ has been defined to satisfy the following orthogonality condition:
$$b(e_i,e_{2n-i+1})=1\ \text{and} \ b(e_i,e_k)=0\  \text{for} \ k\not=2n-i+1\,.$$
 Let $w \in W_I$ be a non-generous permutation. This implies that for some $j\le n$ the sign of $w_j$ is positive.
Let $(w_1,\ldots ,w_{2n})$ be the basis of the $w$-permuted full flag $F_w$ and
suppose that for some $j\le n$ the sign of $w_j$ is not negative.  Thus for every $g\in B_I$
$$
g(w_j)=\lambda (e_k-e_{2n-k+1})+\sum_{\ell <k}c_j(e_\ell-e_{2n-\ell+1})\,.
$$
Observe that $g(w_j)$ is $h$-isotropic and therefore if $j=1$, then, since $\lambda \not=0$, it follows that
$g(F_w)\not\in C_\alpha $ for all $\alpha $.

\bigskip\noindent
Now suppose that $j>1$ and suppose that $g(F_w)$ is in some $C_\alpha $.  This implies in particular that
for every $i<j$ there are linear combinations
$$
\xi_i=g(w_i)+\sum_{\ell < i}c_{i\ell}g(w(\ell))\,.
$$
which when expressed in the standard basis contain only vectors from $V^-$ or $V^+$.
For $g(F_w)$ to be in $C_\alpha $ the same must be true for $g(w_j)$.  This means that there
is a linear combination
$$
\xi_j=g(w_j)+\sum_{i<j} a_i \xi_i
$$
which has the same property.  Hence, $e_k$ or $e_{2n-k+1}$ (or both) must be among the vectors with
non-zero coefficients in the $\xi_i$ for $i<k$.  If both appear, e.g., in $\xi_\kappa $ and $\xi_\delta $,
then $b(\xi_\kappa,\xi_\delta )=1$, contrary to the flag $g(F_w)$ being maximally $b$-isotropic.  If only
one appears, e.g., $e_k$ in $\xi_\kappa $, then $e_{2n-k+1}$ appears in $\xi_j$ and
$b(\xi_\kappa ,\xi_j)=1$ which is again a contradiction. Hence the intersection is empty and this completes the proof.\\

\item[(i)$\Rightarrow$ (iii)] Let $w$ be a generous permutation. Then $w_i<0$ for all $1\leq i\leq n$ and $B_I.(e_j+e_{2n-j+1})$, $j=1,\ldots ,n$, are the orbits of interest. From Remark \ref{rem333} it follows that the orbit $B_I.<e_{i}+e_{2n-i+1}>$ contains the point
$$y=b(e_i+e_{2n-i+1})=\lambda(e_i+e_{2n-i+1})+\ldots  +b_n(e_{n}+e_{n+1})+a_n(e_n-e_{n+1})+\ldots +a_1(e_{1}-e_{2n})$$
where $\lambda\neq 0$. Then we can choose all constants to be $0$ except $\lambda=a_i=\frac{1}{2}$ if the corresponding sign in $\alpha$ is $-$ , then $y=<e_i>=<e_{\tilde{w}_i}>$. If the corresponding sign in $\alpha$ is $+$ , then we can choose all constants to be $0$ except $\lambda=-a_i=\frac{1}{2}$  and therefore $y=<e_{2n-i+1}>=<e_{\tilde{w}(i)}>$. Note that the vectors $e_i$  and $e_{2n-i+1}$ are of negative and positive norms, respectively.

\bigskip\noindent
In summary, we have constructed a set of $T_S$-fixed points in the orbits $S_w$ for
$w$ generous with no two belonging to the same flag domain. By directly checking we see that all possible signatures have been obtained and therefore every flag domain has non-empty intersection with one such $S_w$.
\end{description}
\end{proof}

 \begin{rem}
Assume that we have an intersection point of the Schubert variety $S_w$ with the flag domain $D_\alpha $ given by the flag associated to the basis $(e_{\tilde{w}_1},e_{\tilde{w}_2},...,e_{\tilde{w}_n})$. Therefore this point belongs to the open orbit
$D_{\alpha}$ given by $\alpha=(\alpha_1,\alpha_2,...,\alpha_n)$ where $\alpha_i=+$ if $\tilde{w}_i>0$ and $\alpha_i=-$ if $\tilde{w}(i)<0$. In this case if $\tilde{w}_i<0$, then
$e_{\tilde{w}_i}=e_{2n-w(i)+1}$.\qed
\end{rem}
\begin{rem}
 The basis which defines the flag $F_I$ in the closed orbit of $SP(2n,\mathbb{R})$ is almost the same basis as that in Brecan's thesis \cite{Bre} for the case $SU(p,q)$, and the bases which define the base cycles is
 the same in both cases. Thus in principle Brecan's algorithms and our description for $Sp(2n,\mathbb{R})$ should be similar. In  Brecan's thesis \cite{Bre} for the case $SU(n,2n)$ if the number $2n-i+1$ stays at the left of the number
 $i$ for all $1 \leq i \leq n$ in the one line notation of the permutation $w$, then the Schubert variety $S_w$ has nonempty intersection with some base cycles. If we translate this condition to our style of writing the permutations with negatives, the number $2n-i+1$ will translates to $-i$.
So the condition becomes : If the number $-i$  stays at the left of the number $i$ for all $1 \leq i \leq n$ in the one line notation of the permutation $w$, then the Schubert variety $S_w$ has nonempty intersection with some base cycles. Recall that in our permutation only $i$ or $-i$ appear in  the first half which we used to write the permutations, see $\S 2.2.3$. \qed
\end{rem}


Recall that the base cycle $C_{\alpha}$ is the set
$$C_{\alpha} =\{F \in Z : V_i =(V_i \cap E^-) \oplus ( V_i \cap E^+) , 1 \leq i \leq 2n\}.$$
Since $C_{\alpha}=K/(K \cap B_{z_0})$, where $z_0$ is an appropriate point in an open $G_0$-orbit [i.e. $G_0.z_0$ is open in $Z$], the dimension of $C_{\alpha}$ is $\frac{n(n-1)}{2}$ and Iwasawa-Schubert varieties dual to $C_{\alpha}$ has dimension $\frac{n(n+1)}{2}$.

\bigskip\noindent

The final step in this section is to determine the Schubert varieties of complementary dimension to that of the base cycle $C_{\alpha}$.

\subsection*{The length of the elements of ${\displaystyle S_{n}\ltimes \mathbb{Z}_2^{n}}$}
For dimension computations, let us state how can we compute the length of elements $w \in {\displaystyle S_{n}\ltimes \mathbb{Z}_2^{n}}$ relative to our notation for the Weyl group elements.
\begin{lem}
Fix $w \in {\displaystyle S_{n}\ltimes \mathbb{Z}_2^{n}}$, construct an $\tilde{w} \in {\displaystyle S_{n}\ltimes \mathbb{Z}_2^{n}}$ by the following algorithm:
\begin{enumerate}
	\item Start from left to right in $w$, using simple reflections, place all positive numbers in $w$ step by step in a sequence of $n-$empty boxes beginning from
	the first one in $\tilde{w}$, in the same order as they appeared in $w$.
		\item From left to right in $w$ replace a negative number with its absolute value in the $n-$empty boxes starting from right to left.
\end{enumerate}
  If $\tilde{w}=(\tilde{w}(1),\tilde{w}(2),...,\tilde{w}(n))$, then define $L(\tilde{w})=\frac{n^2-n}{2}-number ~of~ \{\tilde{w}(i): i< k ~and~ \tilde{w}(i) < \tilde{w}(k)\}$
  , and if we have $m$ negative signs in $w$ in the following positions $j_1,j_2,...,j_{m}$, then define $f(w)=\sum_{i=1}^m [(n-j_{i})]$.It follows that the length of $w$ is
  $$l(w)=L(\tilde{w})+ f(w)+m$$
\end{lem}
\begin{proof}
 The length of the permutation $w \in S_{n}\ltimes \mathbb{Z}_2^{n}$ is the minimal number of simple reflection which define $w$. To compute this note that we have two
 kind of reflections , the first $n-1$ reflections are the simple reflections in $S_n$ and the last one is the reflection which flips the sign. This means that if we want
 to flip the sign of $w(j)$, we should move $w(j)$ from its position to the last position and then flip its sign and then return it back to its position. More precisely,
consider the word $w_0=(123...n)$ and let $w$ be the word where $w_{j_i},w_{j_2},...,w_{j_m}$ in the positions $j_1<j_2<...<j_m, ~1\leq j_i \leq n$, are negatives. To
construct $w$ from $w_0$, we first flip the signs for the numbers $w_{j_i},w_{j_2},...,w_{j_m}$. For this purpose we move each number ,starting from $w_{j_m}$,
to the last position and then flip the sign of it. Then we apply the simple reflection to the positive numbers to put them in the same order as they appear in the word $w$. In this way the sum of all these movements and flips is exactly $L(\tilde{w})+m$. The last step is to move each $w_{j_i}$ to its original position in $w$ and denote the total number of these movements by $f(w)$. It then follows that $l(w)=L(\tilde{w})+ f(w)+m$.
 \end{proof}

\begin{ex}
Let $ w=((-12)5(-34) \in S_{5}\ltimes \mathbb{Z}_2^{5}$, then by following the above remark we have
$$(-12)5(-34)\Longrightarrow 254(-1-3)\Longrightarrow 254(31)$$
so $w=(-12)5(-34)$ and $\tilde{w}=254(31)$, then $L(\tilde{w})=7$ and $f(w)=5$. \\
Hence $l(w)=7+5+2=14$
\end{ex}

\begin{defn}\label{Super}
 The element  $w_0=(-n -(n-1)....-2-1) \in S_n \ltimes \mathbb{Z}_2^{n}$ is called \textbf{Super generous permutation} and has length $\frac{n(n+1)}{2}$.
\end{defn}
\bigskip\noindent
\begin{thm}\label{unique}
There exists a unique Iwasawa Schubert variety $S_w$ of dimension $\frac{n(n+1)}{2}$. It is given by the super generous permutation $w_0=(-n -(n-1)....-2-1)$.
\end{thm}
\begin{proof}
Define the set $W_{gen}$ to be the set of all generous permutations, i.e. the set of all permutations  ${\displaystyle w \in S_{n}\ltimes \mathbb{Z}_2^{n}}$ such that all signs
in $w$ are negative. Then by using the above formula to compute the length of $w$ , it is clear that $f(w)+n=\frac{n(n+1)}{2}$ for all $w \in W_{gen}$. Thus the shortest element of $W_{gen}$ has $l(\tilde{w})=0$, i.e. $\tilde{w}$ is the shortest element in $S_n$ which is unique and equal to $(12...n)$. Therefore there is a unique element $w_0 \in W_{gen}$ with length $\frac{n(n+1)}{2}$ which is $(-n,\ldots,-1)$.\\
\end{proof}
\section{Intersection points of Schubert duality}
  Assuming that $w_0\in W_I$ is the super generous permutation, then $S_{w_0}$ is of complementary dimension to the cycles, and the intersection $S_{w_0} \cap C_{\alpha}$ is in fact only one point which is $T_S$-fixed (see Corollary \ref{cor4.2.2}). It should be a simple matter to compute all
 such intersection points.  The argument in the case of $Sp_{2n}(\mathbb{R})$ goes as
 follows.

 \bigskip\noindent
 Since $w=(-n,\ldots ,-1)$, the flag basis corresponding to $w(F_I)$
 $$
 (\hat{v}_n,\ldots \hat{v}_1, v_1,\ldots v_n)
  $$
 where $v_k=e_k-e_{2n-k+1}$ and $\hat{v}_k=e_k+e_{2n-k+1}$.
 \begin{prop}
Let $w$ be a super generous permutation, and let the $\varepsilon $-basis $(\varepsilon_1,\ldots ,\varepsilon_n)$ of $T_S$-eigenvectors be the basis of an intersection point in $S_w \cap C_{\alpha}$, then the $\varepsilon $-basis is given by
 $\varepsilon_k=e_{n-k+1}$ or $e_{n+k-1}$, depending on the signature $\alpha $.
 \end{prop}
 \begin {proof}
 This is a consequence of the following:
 \begin {enumerate}
\item
 $$
 b.\hat{v}_k=\lambda_k(e_k+e_{2n-k+1})+a_k(e_k-e_{2n-k+1})+B_k=K_k+B_k
 $$
 \item
 $\lambda_k\not=0$
 \item
 The intersection $S_w\cap C_\alpha $ is a flag defined by $T_S$-eigenvectors.
 \end {enumerate}
 From the expression for $\hat{v}_k$ it is immediate that all of the possibilities in the statement
 occur.    Furthermore, since the $\lambda_k$ are non-zero, for every $k$ a non-zero contribution from $K_k$
 occurs in the sum
 $$
 \varepsilon_k=\sum_{j\le k} c_{kj}b.v_{w_j}\,.
 $$
 Since $e_{n-k+1}$ and $e_{n+k-1}$ do not occur in $b.v_{w_j}$ for $j<k$, it follows that $\varepsilon_k=K_k+B_k$
 in the standard basis (See $\S 2.2.2$).  Finally, since $\varepsilon_k$ is a $T_S$-eigenvector, it follows that $\varepsilon_k=K_k$ and
 is of the type in the the statement of the proposition.
 \end {proof}
\bigskip\noindent

As a result, if $w$ is super generous, then the set of all intersection points of $(S_w \cap C_{\alpha})$ can be defined to be the following:
$$\text{Supset}(w):=\{F_{\tilde{w}}:\tilde{w} \in W \text{and} ~\tilde{w} \text{~obtained~by~change~none,~some~or~all of~the~numbers}~-i~by~i\}$$
where $F_{\tilde{w}}$ is the maximally $b$-isotropic flag associated to the standard basis $\mathbb{C}^{2n}$ given in (\ref{basisC}). Here $\tilde{w}$ gives us a point in flag domain and to know the signature of this flag domain just replace $i$ by $+$ and $-i$ by $-$.\\

 One more remark, the number of flag domains in $G/B$ is $2^n$. So as a results from the above discussion we have the following corollary
 \begin{cor}\label{cor4.2.2}
If $w$ is the super generous permutation in $W$, then
\begin{enumerate}
   \item The total number of intersection points is $2^n$, each with different signature.
	\item The intersection $S_w \cap C_{\alpha}$ is exactly one $T_S$ fixed point.
\end{enumerate}
 \end{cor}

 \begin{ex}
  The set of all intersection points of the Schubert variety parametrized by the element $w=(-3-1-2)$ is $Supset(w)=\{(-3-1-2).\tilde{z},(3-1-2).\tilde{z},(-31-2).\tilde{z},(-3-12).\tilde{z},(31-2).\tilde{z},(-312).\tilde{z},(3-12).\tilde{z},(312).\tilde{z}\}$ . So $S_w$ intersects the following flag domains: $(---),(+--),(-+-),(--+),(++-),(+-+),(++-),(+++)$.
 \end{ex}


\chapter{Cycle intersection for $SO^*(2n)$}

Here we deal with the homogeneous space $Z$ of maximally isotropic full flags of the complex orthogonal symmetric group $G=SO(2n,\mathbb{C})$ equipped with the action of the real form $G_0=Sp(2n,\mathbb{R})$.


 \section{Conditions for $S_w \cap C_{\alpha} \neq \emptyset$}

\begin{defn}
An element $w=(w_1,w_2,...,w_n) \in S_n \ltimes \mathbb{Z}_2^{n-1}$ is called \textit{\textbf{a dense permutation}}\footnote{Recall that we define a special way of denoting to write the Weyl elements in section $\S 2.2.3$.} if $w_i<0$ for all $1\leq i\leq n$ if $n$ is even, and $w_i<0$ for all $1\leq i\leq n-1$ and $w_n>0$ if $n$ is odd .
\end{defn}
\noindent
This definition will aid in determining the Weyl elements which parametrize  the Schubert varieties that have nonempty intersection with base cycles of flag domains which contain $T_S$-fixed points.

\begin{rem}\label{444}
When discussing the $B_I$-orbits of the base point $F_I$ it is important to explicitly understand the orbits $B_I.(e_i-e_{n+i})$ and $B_I.(e_i+e_{n+i})$. If $F_I$ and $B_I$ are as above, then the orbits of interest are $B_I.(e_i\pm e_{n+i})$, $i=1,\ldots ,n$ if $n$ is even, and $B_I.(e_i\pm e_{n+i})$, $i=1,\ldots ,n-1$ and $B_I.(e_{\frac{n+1}{2}})$ and $B_I.(e_{n+\frac{n+1}{2}})$ if $n$ is odd . In this case the orbits $B_I.(e_i+ e_{n+i})$ , $B_I.(e_i - e_{n+i})$, $B_I.(e_{\frac{n+1}{2}})$ and $B_I.(e_{n+\frac{n+1}{2}})$ have points of the forms
\begin{description}
	\item[ $\bullet$ If $n$ is even:] then for $i=1,\ldots ,n$
$$b(e_i-e_{n+i})=\lambda(e_i-e_{n+i-1})+\ldots  +b_n(e_{n}-e_{2n})+a_n(e_n+e_{2n})+\ldots +a_1(e_{1}+e_{n+1}),$$
$$b(e_i+e_{n+i})=\lambda(e_i+e_{n+i})+a_{i-1}(e_{i-1}+e_{n+i-1})+\ldots +a_1(e_{1}+e_{n+1}), \text{with} \lambda\not=0.$$
\item[ $\bullet$ If $n$ is odd:] then for $i=1,\ldots ,n-1$
$$b(e_i-e_{n+i})=\lambda(e_i-e_{n+i-1})+\ldots  +b_{n-1}(e_{n}-e_{2n})+b_n e_{n+\frac{n+1}{2}}+ a_n e_{\frac{n+1}{2}}+a_{n-1}(e_n+e_{2n})+$$ $~~~~~~~~~~~~\ldots +a_1(e_{1}+e_{n+1}),$
$$b(e_i+e_{n+i})=\lambda(e_i+e_{n+i})+a_{i-1}(e_{i-1}+e_{n+i-1})+\ldots +a_1(e_{1}+e_{n+1}),$$
$$b(e_{\frac{n+1}{2}})=\lambda e_{\frac{n+1}{2}}+ a_{n-1}(e_n+e_{2n})+\ldots +a_1(e_{1}+e_{n+1}),$$
and
$$b(e_{n+\frac{n+1}{2}})=\lambda e_{n+\frac{n+1}{2}}+ a_n e_{\frac{n+1}{2}}+a_{n-1}(e_n+e_{2n})+\ldots +a_1(e_{1}+e_{n+1}),$$
with $\lambda\not=0$.
\end{description}
 Note that in the above orbits, if $b$ is chosen appropriately, then we can arrange $b(e_i-e_{n+i})=e_i$ or $b(e_i-e_{n+i})=e_{n+i}$ for all $i$. This plays a role in the proof of Theorem \ref{thmSO}. \qed
\end{rem}

\bigskip\noindent
Depending on $n$ being odd or even, the base point $F_I\in \gamma^{c\ell}$ is that associated to the following ordered basis:
\begin{description}
	\item[$\bullet$ If $n$ is even:]
	\begin{equation}\label{basisSOeven}
e_1+ie_{n+1},e_2+ie_{n+2}...,e_n+ie_{2n},e_n-ie_{2n},...,e_2-ie_{n+2},e_1-ie_{n+1}.
\end{equation}
\item[$\bullet$ If $n$ is odd:]
\begin{equation}\label{basisSOodd}
e_1+ie_{n+1},...,e_n+ie_{2n},e_{\frac{n+1}{2}},e_{n+\frac{n+1}{2}},e_n-ie_{2n},...,e_1-ie_{n+1}.
\end{equation}
\end{description}
\noindent
and the base point which defines the flag domains and the base cycles is the flag associated to the ordered basis
\begin{equation}\label{basisCS}
e_1,e_2,e_3,......,e_{2n-1},e_{2n}
\end{equation}
These two bases define two maximal tori which are conjugate and the corresponding Weyl groups $W_I$ and $W_{S}$ are isomorphic. \\
Let $\psi$ be the bijective map between  Weyl groups $W_I$ and $W_{S}$ which is defined as follows:
\begin{description}
	\item[$\bullet$ If $n$ is even:] $\psi(\mp (i))=\pm i,$ for all $1\leq i\leq n$.
	\item[$\bullet$ If $n$ is odd:] $\psi(\mp (i))=\pm i,$ for all $1\leq i\leq n-1$ and $\psi(\pm n)=\pm n$.
\end{description}

\begin{prop}
If $w$ is a dense permutation, then the flag $F_{\psi(w)}$ belongs to the orbit $B_I(F_w)$.
\end{prop}
\begin{proof}
Without less of generality assume that $n$ is even. Let $w \in W_I$ be a dense permutation, denote the first $n$ subspaces of the associated maximally isotropic flag $F_w$ by
$$\{0\} \subset <u_{w_1}> \subset <u_{w_1},u_{w_2}> \subset ...... \subset <u_{w_1},...,u_{w_n}>$$
where $u_i=e_{j}\pm e_{n+j}$. We will show that the orbit $B_I(F_w)$ contains the flag $F_{\psi(w)}$.
 For this purpose let $\tilde{w} \in W_{S} $ be the image of $w$ under the bijective map $\psi$ and let $Y_{\tilde{w}}$ be the isotropic flag associated to $\tilde{w}$, i.e., the flag $$Y_{\tilde{w}}:=(\{0\} \subset <e_{\tilde{w}_1}> \subset <e_{\tilde{w}_1},e_{\tilde{w}_2}> \subset ...... \subset <e_{\tilde{w}_1},...,e_{\tilde{w}_n}>)$$
We will show that $Y_{\tilde w} =b(F_w)$ for $b$ appropriately chosen in $B_I$. For this recall that if $w$ is a dense permutation, then $w_i<0$ for all $1\leq i\leq n$
, and therefore $B_I.(e_j-ie_{n+j})$, $j=1,\ldots ,n$, are the orbits of interest. By Remark \ref{444} the orbit $B_I.<e_{j}-ie_{n+j}>$ contains the points
$$y=b(e_j-ie_{n+j})=\lambda(e_j-ie_{n+j})+\ldots  +b_n(e_{n}-ie_{2n})+a_n(e_n+ie_{2n})+\ldots +a_1(e_{1}+ie_{2n})$$
where $\lambda\neq 0$. Choose all constants to be $0$ except $\lambda=a_i=\frac{1}{2}$ , so that $y=<e_i>=<e_{\tilde{w}_i}>$.\\
We repeat this step n-times to construct the first half of the flag $Y_{\tilde{w}}$, and then extend this flag to maximal $h$-isotropic flag, . Therefore, the flag $Y_{\tilde{w}}$ is constructed.
\end{proof}
\begin{thm}\label{thmSO}\textbf{(Dense Permutation Theorem)}.
The following are equivalent
\begin{description}
	\item[~~~~~~~(i)]  $w$ is dense.
  \item[~~~~~~~(ii)]  $B_I(F_w)\cap D_\alpha \not=\emptyset$ for some $\alpha $.
  \item[~~~~~~~(iii)] $B_I(F_w)\cap D_\alpha \not=\emptyset$ for every $\alpha $.
		\end{description}
Furthermore, under any of these conditions, for every $\alpha $ the intersection $B_I(F_w)\cap C_\alpha$
 contains a $T_S$-fixed point.
\end{thm}
\begin{proof}
\begin{description}
	\item[(iii)$\Rightarrow$ (ii)] Trivial direction.
	\item[(ii)$\Rightarrow$ (i)] We handle the case where $n$ is even. The proof for $n$ goes analogously. We assume by contradiction that $w$ is not a dense permutation and will show that $S_w \cap D_{\alpha}=\emptyset$ for all $\alpha $.  In particular, $S_w$ has no $T_S$-fixed points\footnote{See theorem \ref{TM1}}. Recall that the complex bilinear form $b$ is defined by the following conditions:
$$b(e_i,e_{n+i})=1\ \text{and} \ b(e_i,e_k)=0\  \text{for} \ k\not=n+i\,.$$
 Let $w \in W_I$ be a non-dense permutation. Let $(w_1,\ldots ,w_{2n})$ be the basis of the $w$-permuted full flag $F_w$ and
suppose that for some $j\le n$ the sign of $w_j$ is not negative.  Thus for every $g\in B_I$
$$
g(w_j)=\lambda (e_k+ie_{n+k})+\sum_{\ell <k}c_j(e_\ell+ie_{n+\ell})\,.
$$
Observe that $g(w_j)$ is $h$-isotropic and therefore if $j=1$, then, since $\lambda \not=0$, it follows that
$g(F_w)\not\in C_\alpha $ for all $\alpha $.

\bigskip\noindent
Now suppose that $j>1$ and suppose that $g(F_w)$ is in some $C_\alpha $.  This implies in particular that
for every $i<j$ there are linear combinations
$$
\varepsilon_i=g(w_i)+\sum_{\ell < i}c_{i\ell}g(w_\ell)\,.
$$
which when expressed in the standard basis contain only vectors from $V^-$ or $V^+$.
For $g(F_w)$ to be in $C_\alpha $ the same must be true for $g(w_j)$.  This means that there
is a linear combination
$$
\varepsilon_j=g(w_j)+\sum_{i<j} a_i \varepsilon_i
$$
which has the same property.  Hence, $e_k$ or $e_{n+k}$ (or both) must be among the vectors with
non-zero coefficients in the $\varepsilon_i$ for $i<k$.  If both appear, e.g., in $\varepsilon_\kappa $ and $\varepsilon_\delta $,
then $b(\varepsilon_\kappa,\varepsilon_\delta )=1$, contrary to the flag $g(F_w)$ being maximally $b$-isotropic.  If only
one appears, e.g., $e_k$ in $\varepsilon_\kappa $, then $e_{n+k}$ appears in $\varepsilon_j$ and
$b(\varepsilon_\kappa ,\varepsilon_j)=1$ which is again a contradiction. Hence the intersection is empty and this completes the proof.\\

\item[(i)$\Rightarrow$ (iii)] Again, since the two cases are dealt with analogously, we assume here that $n$ is even. Let $w$ be a dense permutation. Then $w_i<0$ for all $1\leq i\leq n$ and $B_I.(e_j-ie_{n+j})$, $j=1,\ldots ,n$, are the orbits of interest. From Remark \ref{444} it follows that the orbit $B_I.<e_{i}+ie_{n+i}>$ contains the point
$$y=b(e_j-ie_{n+j})=\lambda(e_j-ie_{n+j})+\ldots  +b_n(e_{n}-ie_{2n})+a_n(e_n+ie_{2n})+\ldots +a_1(e_{1}+ie_{n+1})$$
where $\lambda\neq 0$. Then we can choose all constants to be $0$ except $\lambda=a_j=\frac{1}{2}$.  It follows that $y=<e_j>=<e_{\tilde{w}_j}>$ if the corresponding sign in $\alpha$ is $-$. If the corresponding sign in $\alpha$ is $+$, then we can choose all constants to be $0$ except $\lambda=-a_j=\frac{1}{2}$ , and $y=<e_{n+j}>=<e_{\tilde{w}_j}>$. Note that the vectors $e_j$  and $e_{n+j}$ are of negative and positive norms, respectively. In summary, we have constructed a set of $T_S$-fixed points in the orbits $S_w$ for $w$ dense with no two belonging to the same flag domain. By directly checking we see that all possible signatures have been obtained and therefore every
flag domain has non-empty intersection with one such $S_w$.
\end{description}
\end{proof}
\noindent
By following the above proof we have the following corollary;

\begin{cor}\label{cor5.2.2}
If $w$ is a dense permutation, then for all $\alpha $ the intersection $S_w \cap C_{\alpha}$ contains exactly one $T_S$-fixed point and in particular is connected.
\end{cor}

 \begin{rem}
Assume that we have an intersection point given by the flag associated to the basis $(e_{\tilde{w}_1},e_{\tilde{w}_2},...,e_{\tilde{w}_n})$. Therefore this point belongs to the open orbit
$D_{\alpha}$ given by $\alpha=(\alpha_1,\alpha_2,...,\alpha_n)$ where $\alpha_j=+$ if $\tilde{w}_j>0$ and $\alpha_j=-$ if $\tilde{w}_j<0$. In this case if $\tilde{w}_j<0$, then $e_{\tilde{w}_j}=e_{n+w_j}$.\qed
\end{rem}

\bigskip\noindent
Recall that the complexification of the maximal compact subgroup $K$ of $G$ has a unique compact orbit in each open $G_0$-orbits which is a complex manifold (see \cite{W1}) and called the base cycle $C_0$.
Similarly as the case $SP(2n,\mathbb{R})$, the base cycle $C_0$ of the flag domain $D_{\alpha}$ of $SO^*(2n)$ is the set
$$C_0 =\{F \in Z : V_i= (V_i \cap E^-) \oplus (V_i \cap E^+) , 1 \leq i \leq 2n\}.$$
where the intersection dimensions are determined by $\alpha $. Since $C_0=K/(K \cap B_{z_0})$ where $z_0$ is any base point in $C_\alpha $, it follows that the dimension of $C_0$ is $\frac{n(n-1)}{2}$ and the Iwasawa Schubert variety must be of dimension $\frac{n(n-1)}{2}$.

\subsection{Length of the elements of ${\displaystyle S_{n}\ltimes \mathbb{Z}_2^{n-1}}$}
For dimension computations, let us state how can we compute the length of elements $w \in {\displaystyle S_{n}\ltimes \mathbb{Z}_2^{n-1}}$ relative to our notation for the Weyl group elements.
\begin{lem}
Fix $w \in W$, construct  $\tilde{w} \in W$ by the following algorithm:
\begin{enumerate}
	\item Start from left to right in $w$, using simple reflections place all positive numbers in $w$ step by step in a sequence of $n-$empty boxes beginning from the first one in $\tilde{w}$, in the same order as they appeared in $w$.
		\item From left to right in $w$ replace a negative number with its absolute value in the $n-$empty boxes starting from right to left.
\end{enumerate}
  If $\tilde{w}=(\tilde{w}_1,\tilde{w}_2,...,\tilde{w}_n)$, then define $L(\tilde{w})=\frac{n^2-n}{2}-\text{the cardinality of } \{\tilde{w}_j: j< k ~and~ \tilde{w}_j < \tilde{w}_k\}$
  , and if we have $2m$ negative signs in $w$ in the following positions $k_1,k_2,...,k_{2m}$, then define $f(w)=\sum_{j=1}^m [(2n-1-k_{2j}-k_{2j-1})]$. It follows that the length of $w$ is
  $$l(w)=L(\tilde{w})+ f(w)+m$$
\end{lem}
\begin{proof}
 The length of the permutation $w \in S_{n}\ltimes \mathbb{Z}_2^{n-1}$ is the minimal number of simple reflections which define $w$. To compute this note that we have two kinds of reflections
 , the first $n-1$ reflections are the simple reflections in $S_n$ and the last one is the reflection which flips the sign. This means that if we want to flip the sign of $w_j$,
we should move $w_j$ from its position to the last position and then flip its sign and then return it back to its position. More precisely,
consider the word $w_0=(123...n)$ and let $w$ be the word where $w_{j_i},w_{j_2},...,w_{j_m}$ in the positions $j_1<j_2<...<j_m, ~1\leq j_i \leq n$, are negatives. To construct $w$ from $w_0$, we first flip the signs for the numbers $w_{j_i},w_{j_2},...,w_{j_m}$. For this purpose we move each number ,starting from $w_{j_m}$, to the last position and then flip the sign of it. Then we apply the simple reflection to the positive numbers to put them in the same order as they appear in the word $w$. In this way the sum of all these movements and flips is exactly $L(\tilde{w})+m$. The last step is to move each $w_{j_i}$ to its original position in $w$ and denote the total number of these movements by $f(w)$. It then follows that $l(w)=L(\tilde{w})+ f(w)+m$.
 \end{proof}

\begin{ex}
Let $ w=((-12)5(-34)6 \in S_{6}\ltimes \mathbb{Z}_2^{5}$, then by following the above proof we have
$$(-12)5(-34)-6 \Longrightarrow 2546(-1-3)\Longrightarrow 2546(31)$$
so $w=(-12)5(-34)6$ and $\tilde{w}=2546(31)$, then $L(\tilde{w})=9$ and $f(w)=6$.\\
Hence $l(w)=9+6+1=16$.
\end{ex}

\bigskip\noindent
\begin{defn}
Let  $w=(-n -(n-1)....-2-1) \in S_n \ltimes \mathbb{Z}_2^{n-1}$ if $n$ is even  and $w=(n -(n-1)....-2-1) \in S_n \ltimes \mathbb{Z}_2^{n-1}$ if $n$ is odd. any of these two elements called  \textbf{Super dense} and has dimension $\frac{n(n-1)}{2}$.
\end{defn}

\bigskip\noindent

As the case of the real form $SP(2n,\mathbb{R})$, by direct calculation of the dimension of the dense permutations we get the following  result.
\bigskip\noindent

\begin{thm}
There exists a unique Iwasawa Schubert variety $S_w$ of dimension $\frac{n(n-1)}{2}$. It is given by the super dense permutation $\tilde{w}=(-n -(n-1)....-2-1)$ if $n$ is even and by $\tilde{w}=(n -(n-1)....-2-1)$ if $n$ is odd.
\end{thm}
\begin{proof}
Assume firstly that n is even, then consider the set $W_{den}$ to be the set of all dense permutations, i.e. the set of all permutations  ${\displaystyle w \in S_{n}\ltimes \mathbb{Z}_2^{n-1}}$ such that all signs in $w$ are negative. Then, using our formula to compute the length of $w$ it is clear that $f(w)+n=\frac{n(n-1)}{2}$ for all $w \in S_{n}\ltimes \mathbb{Z}_2^{n-1}$. Thus the shortest element of $W_{den}$ has $l(\tilde{w})=0$, i.e. $\tilde{w}$ is the shortest element in $S_n$ which is unique and equal to $(12...n)$. Therefore there is a unique element $w \in W_{den}$ with length $\frac{n(n-1)}{2}$ which is $(-n -(n-1)....-2-1)$.\\

Secondly, let $n$ is odd and let $\tilde{W}_{den}$ to be the set of all dense permutations, i.e. the set of all permutations  ${\displaystyle w \in S_{n}\ltimes \mathbb{Z}_2^{n-1}}$ such that all signs in $w$ are negative except the sign of the entry $n$ is positive. Then, using our formula to compute the length of $w$, it is clear that $f(w)+n=\frac{n(n-1)}{2}$ for all $w \in S_{n}\ltimes \mathbb{Z}_2^{n-1}$ only if the entry $n$ sitting in the first box of the permutation. Thus shortest element of $\tilde{W}_{den}$ has $l(\tilde{w})=0$. Then the shortest element in $S^n$ is $\tilde{w}$ which is unique and equal to $(12...n)$ which implies that we have a unique element $w \in \tilde{W}_{den}$ with length $\frac{n(n-1)}{2}$ which is $(n -(n-1)....-2-1)$.\\
\end{proof}
\section{Intersection points of Schubert duality}
 Assuming that the element $w_0\in W_I$ is $w_0=(-n -(n-1)....-2-1)$ if $n$ is even and  $w_0=(n -(n-1)....-2-1)$ if $n$ is odd , we have shown that $S_w$ is of complementary dimension to the cycles, and the intersection $S_w \cap C_{\alpha}$ is  in fact only one $T_S$-fixed point (See corollary \ref{cor5.2.2}). Here we describe all intersection points.  The argument in the case of $SO^*(2n)$ , which is similar to that in the case $SP(2n,\mathbb{R})$, goes as follows.

 \bigskip\noindent
 Since $w=(-n,\ldots ,-1)$, the flag basis corresponding to $w_0$
 $$
 (\hat{v}_n,\ldots \hat{v}_1, v_1,\ldots v_n)
  $$
 where $v_k=e_k+ie_{n+k}$ and $\hat{v}_k=e_k-ie_{n+k}$.
 \begin{prop}
If $w$ is a super dense permutation, and the $\varepsilon $-basis $(\varepsilon_1,\ldots ,\varepsilon_n)$ of $T_S$-eigenvectors defines the intersection point in $S_w \cap C_{\alpha}$, then $\varepsilon_k=e_{k}$ or $e_{n+k}$, depending on the signature $\alpha $.
 \end{prop}
 \begin {proof}
 This is a consequence of the following:
 \begin {enumerate}
\item
 $$
 \hat{v}_k=\lambda_k(e_k-e_{n+k})+a_k(e_k+ie_{n+k})+R_k=K_k+R_k
 $$
 \item
 $\lambda_k\not=0$
 \item
 The intersection $S_w\cap C_\alpha $ is a flag defined by $T_S$-eigenvectors.
 \end {enumerate}
 From the expression for $\hat{v}_k$ it is immediate that all of the possibilities in the statement
 occur.    Furthermore, since the $\lambda_k$ are non-zero, for every $k$ a non-zero contribution from $K_k$
 occurs in the sum
 $$
 \varepsilon_k=\sum_{j\le k} c_{kj}b(w(j))\,.
 $$
 Since $e_{k}$ and $e_{n+k}$ do not occur in $b(w(j))$ for $j<k$, it follows that $\varepsilon_k=K_k+E_k$
 in the standard basis (See $\S 2.2.2$).  Finally, since $\varepsilon_k$ is a $T_s$-eigenvector, it follows that $\varepsilon_k=K_k$ and
 is of the type in the the statement of the proposition.
 \end {proof}

\bigskip\noindent
If $w$ is a super dense permutation, then the set of all intersection points of $S_w$ is denoted by  $Supset_e(w)$ if n is even and by $Supset_o(w)$ if n is odd, where
 $$Supset_e(w):=\{\tilde{w}(F_S) \in W: \tilde{w} ~obtained~by~change~none,~some~or~all~of~the~numbers~-j~by~j\}$$
 $$Supset_o(w):=\{\tilde{w}(F_S) \in W: \tilde{w} ~obtained~by~change~none,~some~or~all~of~the~numbers~-j~by~j$$
$$~or~n~by~-n\},~~~~~~~~~~~~~~~~~~~~~~~~~~~~~~~~~~~~~~~~~~~~~~~~~~~~$$
 where $F_S$ is the maximally $b$-isotropic flag associated to the standard basis $e_1,e_2,...,e_{2n}$. Here $\tilde{w}(F_S)$ gives us a point in an open orbit $D_{\alpha}$ and to know this open orbit just replace $i$ by $+$ and $-i$ by $-$.


\chapter{Cycle intersection for $SO(p,q)$}
In this chapter we study the case of the real form $SO(p,q)$ of $SO_{m}(\mathbb{C})$. The general results here are stated in terms of algorithms (See definitions \ref{de1} and \ref{de2}); in fact it seems impossible to avoid this. In corollary \ref{cor12.4} we give concrete formulas for the intersection points in
$S_w\cap C_\alpha $ if the intersection non-empty and $S_w$ is of complementary dimension . Also, the number of intersection points with $C_0$ is explicitly computed in Theorem \ref{thm5.18}.

\section{Conditions for $S_w \cap C_{\alpha} \neq \emptyset$}
In the present section we describe the conditions for an element $w$ of the Weyl group
to parametrize Schubert variety with $S_w$ must satisfy in order that $S_w\cap D_\alpha \not=\emptyset $ for some flag domain $D_\alpha $. As would be expected, a special type of permutation plays a fundamental role.


\begin{defn}\label{de1}
An element $w \in W$ is called \textit{\textbf{a harmonic permutation}} if it satisfies the following conditions:\\
\textbf{If $q$ is even:} The number $-(2i-1), 1\leq i \leq \frac{q}{2}$, sits in any place to the left of the number $(2i)$ or $(-2i), 1\leq i \leq \frac{q}{2}$, in the one line
notation of the permutation and the order of the numbers $q+i$ or $-(q+i)$, where $1\leq i \leq p-q$ is arbitrary. \\
\textbf{If $q$ is odd:} The number $-(2i-1), 1\leq i \leq \frac{q-1}{2}$, sits in any place to the left of the number $(2i), 1\leq i \leq \frac{q-1}{2}$, and the number $-q$ sits
in the last position in the one line notation of the permutation, and the order of the numbers $q+i$ or $-(q+i)$, where $1\leq i \leq \frac{p-q}{2}$ is arbitrary.
\end{defn}

 \begin{ex}
In $SO(4,2)$ the relevant pairs are $(-12)$ and $(-1-2)$. As a result we have $6$ harmonic permutations. These are: $(-12-3),(-1-32),(-3-12),(-1-23),\\(-13-2),(3-1-2)$.
\end{ex}

Recall that the fixed point in the closed orbit $\gamma^{cl}$ is the flag associated to the following ordered basis
\begin{itemize}
	\item If $m$ is even, then the basis is   \begin{eqnarray}\label{basis 1}
 e_1+e_{2q},...,e_q+e_{q+1},e_{2q+1}+ie_{2q+2},e_{2q+3}+ie_{2q+4},...,e_{2n-1}+ie_{2n},&&  \nonumber \\
 e_{2n-1}-ie_{2n},...,e_{2q+3}-ie_{2q+4},e_{2q+1}-ie_{2q+2},e_q-e_{q+1},...,e_1-e_{2q}.& &
\end{eqnarray}
\item If $m$ is odd, then the basis is   \begin{eqnarray}\label{basis 2}
e_1+e_{2q},e_2+e_{2q-1},...,e_q+e_{q+1},e_{2q+1},e_{2q+2}+ie_{2q+3},e_{2q+4}+ie_{2q+5},...,e_{2n}+ie_{2n+1},&& \nonumber \\
e_{2q+1},e_{2n}-ie_{2n+1},...,e_{2q+4}-ie_{2q+5},e_{2q+2}-ie_{2q+3},e_q-e_{q+1},...,e_2-e_{2q-1},e_1-e_{2q}~~~&&.
\end{eqnarray}
\end{itemize}
\begin{rem}\label{rem5.4}
For $v_i$ any such basis vector and $b\in B_I$ the form of $b.v_i$ is given as follows:
\begin{enumerate}
	\item $b.(e_i-e_{2q-i+1})=\lambda(e_i-e_{2q-i+1})+\ldots  +b_n(e_{q}-e_{q+1})+a_n(e_q+e_{q+1})+\ldots +a_1(e_{1}+e_{2q})$,
	\item $b.(e_i+e_{2q-i+1})=\lambda(e_i+e_{2q-i+1})+a_{i-1}(e_{i-1}-e_{2q-i})+\ldots +a_1(e_{1}-e_{2q})$,
	\item $b.(e_{2q+2j-1}-ie_{2q+2j})=\lambda(e_{2q+2j-1}-ie_{2q+2j})+\ldots  +b_m(e_{m-1}-e_{m})\\
	~~~~~~~~~~~~~~~~~~~~~~~~~~~~~~+a_m(e_{m-1}+e_{m})+\ldots +a_1(e_{1}+e_{2q})$ \\
	
	where $m=2n$ or $2n+1$.
	\item $b.(e_{2q+2j-1}+ie_{2q+2j})=\lambda(e_{2q+2j-1}+ie_{2q+2j})+a_{i-1}(e_{2q+2j+1}+ie_{2q+2j+2})+\ldots +a_1(e_{1}+e_{2q})$,
	\item $b.(e_{2q+1})=\lambda(e_{2q+1})+a_{n}(e_{2n}+ie_{2n+1})+\ldots +a_1(e_{1}+e_{2q})$, if $m=2n+1$,
\end{enumerate}
 with $\lambda\not=0$ in all cases above. Note that in the above orbits, if $b$ is chosen appropriately, then we can arrange all linear combination for every one of the above vectors to be in the standard basis of $T_S$-eigenvectors. See Theorem \ref{thmSOPQ}.\qed
\end{rem}

\bigskip\noindent
In the following result $\psi :W_I\to W_S$ denotes the bijective map between Weyl groups which was introduced in Chapter 2.
	

	\begin{prop}\label{lemSO}
If $w$ is a  a harmonic permutation, then the flag $F_{\psi(w)}$ belongs to the orbit $B_I(F_w)$.
\end{prop}

\begin{proof}
 We handle the case where $m=2n$. The proof for $m=2n+1$ goes analogously. Let us first prove the theorem for the case that $q$ is even.
 For this let $w \in W_I$ be a harmonic permutation and define $F_w=w.(F_I)$ to be the isotropic full flag associated to $w$. Denote by
$$\{0\} \subset <u_{w_1}> \subset <u_{w_1},u_{w_2}> \subset ...... \subset <u_{w_1},...,u_{w_n}>$$
the first $n$ subspaces of $F_w=w(F_I)$, where $u_{w_i}$ is a vector from the basis above.
Let $\tilde{w} \in W_{S} $ be the image of $w$ under the bijective map $\psi$ and let $Y_{\tilde{w}}$ be the isotropic flag associated to $\tilde{w}$ such that the first half of $Y_{\tilde{w}}$ is
$$\{0\} \subset <\varepsilon_{\tilde{w}_1}> \subset <\varepsilon_{\tilde{w}_1},\varepsilon_{\tilde{w}_2}> \subset ...... \subset <\varepsilon_{\tilde{w}_1},...,\varepsilon_{\tilde{w}_n}>$$
Our claim here is that this flag is an intersection point in $B.F_w\cap C_\alpha $.
To prove this we will construct $b$ with $b(F_w)=Y_{\tilde w}$. From the definition of harmonic permutation, there are two possibilities for $w_1$: $|w_1|=2i-1\leq q$ or $|w_1|>q$.\\
 \textbf{Case 1:} ~If $|w_1|>q $. Consequently then $u_{w_1}=v_{\tilde{w}_1}$, so the orbit $B_I.<u_{w_1}>$ contains the point $<\varepsilon_{\tilde{w}_1}>$.\\
 \textbf{Case 2:} ~If $|w_1|=2i-1\leq q $, then we must consider the orbit $B_I.<e_{2\tilde{w}_1-1}-e_{2q+w_1+1}>$.
 By Remark (\ref{rem5.4}) above, the orbit $B_I.<e_{2\tilde{w}_1-1}-e_{2q+w_1+1}>$ contains points of the form $$y=<\alpha_1 (e_{2\tilde{w}_1-1}+e_{2q+w_1+1})+ \alpha_2 (e_{2\tilde{w}_1}+e_{2q+w_1})+\alpha_3 (e_{2\tilde{w}_1}-e_{2q+w_1})+\alpha_4 (e_{2\tilde{w}_1-1}-e_{2q+w_1+1})>$$ where $\alpha_1=\pm \alpha_4$ and $\alpha_2=\pm \alpha_3$.
By taking $\alpha_1=\alpha_4=\frac{1}{2}$ and $\alpha_2=\alpha_3=\frac{1}{2}i$, it follows that $y=<e_{2\tilde{w}_1-1}+ie_{2\tilde{w}_1}>=<v_{\tilde{w}_1}>$.\\

To construct the $j$-vector of $b(v_{w_j})$ to obtain the subspace $V_{\tilde{w}_j}$ we must consider three cases:\\
\textbf{Case 1:} If $|w_j|>q$, then $w_j=\tilde{w}_j$ and $\varepsilon_{\tilde{w}_j}=u_{w_j}$, so the orbit $B_I.<u_{w_j}>$ contains the point $\varepsilon_{\tilde{w}_j}$.\\
\textbf{Case 2:} If $|w_j|=2i-1\leq q$, then our job goes through the orbit $B_I.<e_{2\tilde{w}_j-1}-e_{2q+w_j+1}>$.
 By using Remark (\ref{rem5.4}) we see that the orbit $B_I.<e_{2\tilde{w}_j-1}-e_{2q+w_j+1}>$ contains points of the form $$y=<\alpha_1 (e_{2\tilde{w}_j-1}+e_{2q+w_j+1})+\alpha_2 (e_{2\tilde{w}_j}+e_{2q+w_j})$$ $$+\alpha_3 (e_{2\tilde{w}_j}-e_{2q+w_j})+\alpha_4 (e_{2\tilde{w}_j-1}-e_{2q+w_j+1})>$$ where $\alpha_1=\pm \alpha_4$ and $\alpha_2=\pm \alpha_3$.
  By taking $\alpha_1=\alpha_4=\frac{1}{2}$ and $\alpha_2=\alpha_3=\frac{1}{2}i$, it follows that $y=<e_{2\tilde{w}_j-1}+ie_{2\tilde{w}_j}>=<\varepsilon_{\tilde{w}_j}>$. Therefore $b$ is constructed to obtain the flag
  $$\{0\} \subset <\varepsilon_{\tilde{w}_1}> \subset <\varepsilon_{\tilde{w}_1},\varepsilon_{\tilde{w}_2}> \subset ...... \subset <\varepsilon_{\tilde{w}_1},...,\varepsilon_{\tilde{w}_j}>.$$ \\
  \textbf{Case 3:} If $|w_j|=2i$ , then the orbit is $B_I.<e_{|w_j|}+e_{|2 \tilde{w}_j|-1}>$ is relevant. In this case the points
$$y=<\alpha_1 (e_{|w_j|-1}+e_{|2 \tilde{w}_j|}) +\alpha_2 (e_{|w_j|}+e_{|2 \tilde{w}_j|-1})+ \alpha_3 (e_{|w_j|-1}+ie_{|w_j|})>$$
belong to the orbit $B_I.<e_{|w_j|}+e_{|2 \tilde{w}_j|-1}>$. For $\alpha_1=-i, \alpha_2=1$ and $\alpha_3=i$ we have $y=<e_{|2 \tilde{w}_j|-1}+ie_{|2 \tilde{w}_j|}>$.
Therefore, the $j$-vector of $b$ is constructed in this case as well..\\
Thus by induction we observe that $b\in B_I$ can be constructed with $b(F_w)=Y_{\tilde{w}}$.\\

We complete the proof by handling the case where $q$ is odd.
For this point we can repeat the steps of the proof above for $w_j>q$ and $|w_j|=2i-1$. For the case where $|w_j|=2i$, we apply the same method as above, only changing
$2|\tilde{w}_j|$ by $2(|\hat{w}_j|+1)$. If $w_n=-q$, then $\hat{w}_j=-n$. In this case the orbit of relevance is $B_I.<e_q-e_{q+1}>$. As we see from Remark \ref{rem5.4} $y=\alpha_1 (e_q+e_{q+1})+ \alpha_2 (e_q-e_{q+1})$ belongs to this orbit . The desired result is then obtained by taking $\alpha_1=1,~ \alpha_2=-1$ the point
$y=e_{q+1}$ belongs to the orbit $B_I.<e_q-e_{q+1}>$. Therefore the element of $B_I$ is also constructed in the case where $q$ is odd.\\
\end{proof}

\begin{thm}\label{thmSOPQ}\textbf{(Harmonic Permutation Theorem)}.
The following are equivalent
\begin{description}
	\item[~~~~~~~(i)]  $w$ is harmonic.
  \item[~~~~~~~(ii)]  $B_I(F_w)\cap D_\alpha \not=\emptyset$ for some $\alpha $.
		\end{description}
Under either of these conditions for every $\alpha$ the nonempty intersection $B_I(F_w)\cap C_\alpha$
 contains a $T_S$-fixed point.
\end{thm}
\begin{proof}
\begin{description}
	\item[(i)$\Rightarrow$ (ii)] Is exactly the statement of Proposition \ref{lemSO}.
	\item[(ii)$\Rightarrow$ (i)] Assuming that $w$ is not harmonic permutation, we will show that $S_w \cap D_{\alpha} = \emptyset,$ for all $\alpha$, i.e. $S_w$ has no $T_S$-fixed points\footnote{See Corollary \ref{fixthm}}. Assume to the contrary that there exists $b \in B_I$ such that $b.(F_w) \in S_w \cap C_{\alpha}$. For $b(F_w)$ to be fixed, $b$ has to have a certain shape and at each stage where the condition of harmonicity is violated, we should prove that there is no such $b$. For this recall that the complex bilinear form $b$ has been defined to satisfy the following orthogonality condition:
$$b(e_j-e_{2q-j+1},e_j+e_{2q-j+1})=1,~1\leq j \leq q \ \text{and} \ b(e_j \pm e_{2q-j+1},e_k \pm e_{2q-k+1})=0\  \text{for} \ k\not=j\,.$$ Since $w$ is not a harmonic permutation, there exist a pair of the form $(-(2i-1),2i)$ or of the form
$(-(2i-1),-2i)$, where $1\leq i \leq \frac{q}{2}$, such that $\mp 2i$ sits to the left of $-(2i-1)$. Assume that $w_j=2i$ is the first even number which appears in $w$ such that
$2i$ sits to the left of $-(2i-1)$. Then $b$ does in fact yield a $T_S$-fixed partial flag,i.e.
\begin{equation}\label{1}
 \{0\} \subset <\varepsilon_{\tilde{w}_1}> \subset <\varepsilon_{\tilde{w}_1},\varepsilon_{\tilde{w}_2}> \subset ...... \subset <\varepsilon_{\tilde{w}_1},...,v_{\tilde{w}_{j-1}}>
\end{equation}
Now we check that there is no $b\in B_I$ so that $b.(v_{w_j})$ defines an extended partial flag which is $T_S$-fixed. For this we consider  the orbit $B_I.<e_{|w_j|}+e_{|2 \tilde{w}_j|-1}>$ which contains points of the form
\begin{equation}\label{2}
\alpha_1 (e_1+e_{2q})+\alpha_2 (e_2+e_{2q-1}) + ......+\alpha_j (e_{|w_j|}+e_{|2 \tilde{w}_j|-1})
\end{equation}
This is a linear combination of $h$-isotropic vectors for all $\alpha_i$. Recall that the flag of intersection has non $h$-isotropic vectors and to have one of these vectors in this step which is linearly independent with all vectors in the flag (\ref{1}) above , we should add the point $(e_{|w_j|-1}+ie_{|w_j|})>$ from the flag (\ref{1}) to the linear combination in (\ref{2}), but this vector is not in the flag (\ref{1}). Thus, as was claimed, no $b\in B_I$ has that property that $b(F_w)$ is $T_S$-fixed. \\
If $w_j=-2i$ is the first even number such that $-2i$ sits to the left of $-(2i-1)$, then the relevant orbit is $B_I.<e_{|w_j|}-e_{|2 \tilde{w}_j|-1}>$ which contains points of
the form  $$\alpha _1 (e_1+e_{2q})+...+\alpha_q (e_q+e_{q+1})+\alpha_{q+1} (e_q-e_{q+1})+...+\alpha_s (e_{|w_j|}-e_{|2 \tilde{w}_j|-1}).$$ To have non-isotropic point which is linearly independent with all points in the flag (\ref{1}) , we should add any of the points
$(e_{|w_j|-1}+ie_{|w_j|})>$, $(e_{|w_j|-1}-ie_{|w_j|})>$, $(e_{2q-|w_j|+1}+ie_{2q-|w_j|+2})>$ or $(e_{2q-|w_j|+1}-ie_{2q-|w_j|+2})>$
from the flag (\ref{1}) to the linear combination in (\ref{2}). But the flag (\ref{1}) does not contain any of these points.So again, for all $b\in B_I$ the flag $b(F_w)$ is not $T_S$-fixed and consequently $S_\alpha \cap D_\alpha =\emptyset $.
\end{description}
 \end{proof}

\section{Introduction to the combinatorics}

For the remainder of this chapter we only discuss the intersection properties of the Iwasawa-Schubert cells which are of complementary dimension to $C_0$. Recall that the maximal compact subgroup of $SO(p,q)$ is $K_0 := S(O(p) \times O(q))\subset S(U(p) \times U(q))$. For $E^+$ and $E^-$ as in $\S 2.4.2$, the base cycle in the flag domain $D_{a,b}$ is given by
$$C_0 =\{F \in Z : \text{dim} V_i \cap E^-=\sum_{j=1}^i a_j~ and~ \text{dim} V_i \cap E^+ =\sum_{j=1}^i b_j , 1 \leq i \leq m\}.$$
If $m=2n$, we have two cases: If q is even, the dimension of $C_0$ is $\frac{p(p-2)}{4}+\frac{q(q-2)}{4}$. If q is odd, then $C_0$ has the dimension $\frac{(p-1)^2}{4}+\frac{(q-1)^2}{4}.$ If $m=2n+1$, then the dimension of the base cycle is $\frac{(p-1)^2}{4}+\frac{q(q-2)}{4}$ if q is even and is $\frac{p(p-2)}{4}+\frac{(q-1)^2}{4}$ if q is odd.
  Since we have restricted to the case where $Z=G/B$ is the manifold of complete flags, it follows that $\text{dim}~S_w=\frac{pq}{2}$ if $q$ and $n$ are even, and $\text{dim}~S_w=\frac{pq-1}{2}$ if $q$ or $n$ are odd.
\begin{defn}\label{de2}
A harmonic permutation $w \in S_n \ltimes \mathbb{Z}_2^{n-1}$ is called \textit{\textbf{a perfect harmonic permutation}} if it is constructed by the following algorithm:
\begin{description}
	\item[A. ] Start with a sequence of $n$ empty boxes which are to be filled in order to construct $w$.
  \item[B. ] Consider the pairs $(-(2j-1),2j)$ and $(-(2j-1),-2j)$ for all $1\leq j \leq \frac{q}{2}$.\\

\item[C. ] If $q$ is even:
\begin{enumerate}
\item The pairs in step B are $(-1,2),(-1,-2),(-3,4),...,(-(q-1),q),(-(q-1),-q)$.
 \item Step by step, starting from $(-1,2)$ until $(-(q-1),-q)$. For each $1\leq j \leq \frac{q}{2}$ we have two pairs of the forms $(-(2j-1),2j)$ and $(-(2j-1),-2j)$, so choose only one pair of them for each step and omit the other form the above list.
 \item If we choose the pair of the form $(-(2j-1),2j)$, place this pair in any box in $w$ such that the components $-(2j-1)$ and $2j$ of this pair sits as close as possible to each other.
 \item If we choose the pair of the form $(-(2j-1),-2j)$, place this pair in any box in $w$ such that the components $-(2j-1)$ and $-2j$ of this pair sits as close as possible to each other
 and all pairs $(-(2i-1),2i)$ or $(-(2i-1),-2i)$ with $i>j$ are sitting to the left of this pair and the pairs of the form $(-(2i-1),-2i)$ with $i<j$ are sitting in a decreasing order with respect to $i$ to the right of $(-(2j-1),-2j)$.
\item  Once a pair is placed its position can be ignored so that the places at the immediate left and right of this pair become adjacent.
 \item After all pairs are placed, the remaining numbers $\pm (q+i)$ for $1\leq i \leq \frac{p-q}{2}$ are placed in the remaining spots in the strictly increasing order	with respect to $|w(i)|$ such that all number $\pm(q+i)$ for $1\leq i \leq \frac{p-q}{2}$ are  sitting to the left of the pairs $(-(2i-1),-2i)~~1\leq i \leq \frac{q}{2}$ .
	If the number of negative signs in all pairs from steps 2 and 3 is even, then all numbers of the reminder numbers are positive, and if the number of negative signs is odd, then all numbers of the reminder numbers are positive except the number $n$ is negative.\\
\end{enumerate}
\item[D. ] If $q$ is odd:
 \begin{enumerate}
\item The pairs in step B are $(-1,2),(-3,4),(-5,6),...,(-(q-2),q-1),-q$.
\item In the last box of $w$ put the number $-q$.
 \item Step by step, starting from $(-1,2)$ until $(-(q-2),q-1)$, choose a pair and place this pair in any box of the first $n-1$ boxes in $w$ such that the components $-(2j-1)$ and $2j$ of this pair sit as close as possible to each other. This means that once a pair is placed it can be ignored so that the places at the immediate left and right of this pair become adjacent.
 \item After all pairs are placed, the numbers $\pm(q+i)$ for $1\leq i \leq \frac{p-q}{2}$ are placed in the remaining spots in the stricly increasing order with respect to $|w(i)|$  . If the number of negative signs in all pairs from step 1 and 2 is even, then all numbers of the reminder are positive, and if the
	number of negative signs is odd then all remaining numbers of the reminder are positive expect the number $n$ is negative.
\end{enumerate}
\end{description}
\end{defn}
\bigskip\noindent
\begin{rem}\label{rempo}
If $q$ is odd, then the signature of the flag domain $D_{\alpha}$ has $+$ in the last position.\qed
\end{rem}
\bigskip\noindent
\begin{rem}
If $W= S_n \ltimes \mathbb{Z}_2^{n}$, then a perfect harmonic permutation consists only of pairs of the form $(-(2j-1),2j)$ for all $1\leq j \leq \frac{q}{2}$.  It is constructed as above, in particular such that the sign of $n$ is $+$.\qed
\end{rem}
\bigskip\noindent

\begin{ex}
If $p=6,~q=4$, then the perfect harmonic permutations are:\\$(-12)(-34)5,(-12)5(-34),5(-12)(-34),(-3~4)(-1~2)5,(-3~4)5(-1~2),5(-3~4)(-1~2),\\(-12)-5(-3-4) ,-5(-12)(-3-4),-5(-3-4)(-12),(-34)-5(-1-2), -5(-34)(-1-2),\\5(-3-4)(-1-2), (-3-1~2~4)5,5(-3-1~2~4),-5(-3-12-4)$.
\end{ex}
\bigskip\noindent

\begin{ex}
If $p=10,~q=6$, then the element $78(-5~6)(-3-4)(-1~2)$ is a perfect harmonic permutation while the element $78(-3-4)(-1~2)(-5~6)$ is not a perfect harmonic permutation.
\end{ex}

\bigskip\noindent
Recall that the dimension of the cells corresponds to the length of the word $w$ in the Weyl group, i.e. if $F_w=w.F_I$, then $dim(B.F_w)=l(w)$. So if we want to discuss the dimension of the Schubert cell it is enough to discuss the length of Weyl elements. \\
\begin{prop}\label{lem5.13}
Every perfect harmonic permutation $w \in {\displaystyle S_{n}\ltimes \mathbb{Z}_2^{n-1}}$ has dimension $\frac{pq}{2}$ if $q$ is even and $\frac{pq-1}{2}$ if $q$ is odd.
\end{prop}
\begin{proof}
Given  a perfect harmonic permutation $w$ we consider three cases.These depend on which of the pairs $(-1~2), (-3~4), (-1 -2)$ or $(-3 -4)$  is contained in $w$. In each of these cases our proof goes by induction on the dimension of the flag manifold.
Without loss of generality let $p>q\geq 6$ because if $q \leq 6$ then the permutaion has only one or tow pairs and the proof become trivial.\\

\textbf{Case 1:} The permutation $w$ contains the pairs $(-1~2)$ and $(-3~4)$.\\
First, remove the pairs $(-1~2)$ and $(-3~4)$ from $w$ to have a new permutation
$v$ consisting of the numbers $\{5,6,...,n\}$. Define a function $f:\{5,...,n\}\longrightarrow \{1,...,n-4\}$ by $f(i)=i-4$ for all $ 1\leq i \leq n-4$. This is a bijective map which sends $v$ to
$\hat{w} \in {\displaystyle S_{n-4}\ltimes \mathbb{Z}_2^{n-5}} $. Note that $\hat{w}$ is a perfect harmonic permutation in $ {\displaystyle S_{n-4}\ltimes \mathbb{Z}_2^{n-5}} $. Thus by the induction assumption  $l(\hat{w})=\frac{(p-4)(q-4)}{2}$. Since $v=f^{-1}(\hat{w})$ it follows that v has the same length as $\hat{w}$. So we put the numbers $1234$ to the left of $v$ to have an element
$\tilde{w}  \in {\displaystyle S_{n}\ltimes \mathbb{Z}_2^{n-1}}$ with length $ \frac{(p-4)(q-4)}{2}$.
To split the sign of 1 and 3 (i.e., to change the postive sign to negative sign) we add $n-4+1$ to $ \frac{(p-4)(q-4)}{2}$ to send $3$ to the last position and add $n-4+2$ to send 1 to the position before the last position.
Consequently, we have the element $(24)v(13)$ with length $ \frac{(p-4)(q-4)}{2}+n-4+1+n-4+2=\frac{(p-4)(q-4)}{2}+2n-5$ and it follows that the element $(24)v(-3-1)$ has length $\frac{(p-4)(q-4)}{2}+2n-4$.
We then return to the original $w$ and remove only the pair $(-1~2)$. Then we define $g$ to be the number of positions to the left of $(-34)$.  As a result we have $n-4-g$ positions to
left of $(-3~4)$ and $-3$ should cross $n-4-g+1$ positions to end up in the last position and $4$ should cross $g$ boxes to end up in the first position.

 \bigskip\noindent
Finally we return to the original $w$ and define $h$ to be the number of positions to the left of $(-1~2)$ and $n-2-f$ to the right. In this situation $-1$ must cross $n-2-h+1$ positions to end up in
the last position and $2$ must cross $g$ positions to end up in the first position. It follows that the length of $w$ is  $\frac{(p-4)(q-4)}{2}+2n-4+n-4-g+1+g+n-2-h+1+h=\frac{pq}{2}$.\\
If $q$ is odd, then the length of $\hat{w} \in {\displaystyle S_{n-4}\ltimes \mathbb{Z}_2^{n-5}} $ is $l(\hat{w})=\frac{(p-4)(q-4)-1}{2}$. After applying the same steps as those above it follows that $l(w)=\frac{pq-1}{2}$\\

\textbf{Case 2:} The permutation $w$ contains the pair $(-1~-2)$.\\
Note that this case only occurs if q is even. Since $w$ is a perfect harmonic permutation, the pair $(-1 -2)$ sits in the last 2 positions of $w$. Remove the pair $(-1 -2)$ from $w$ to obtain $\hat v$.  Since $q$ is even, a similar argument to that above shows that  $l(\hat v)=\frac{(p-2)(q-2)}{2}$ because q is even . Put the pair $(1~2)$ to the left of $\hat v$.  It follows that $\tilde v=(1~2)\hat v$
has the same length as $\hat v$. To split the sign of 1 and 2, each of them must cross $n-2$ positions. Having made this move we then apply $2$ reflection to obtain  the pair $(-1-2)$ in the last
two positions. Hence the length of $w$ become $l(w)=\frac{(p-2)(q-2)}{2}+2(n-2)+2=\frac{pq}{2}$.\\

\textbf{Case 3:} The permutation $w$ contains the pairs $(-1~2)$ and $(-3~-4)$.\\
Note that this case appear only if q is even. In this case these two pairs appear in $w$ in the following forms: $(-3-4-1~2)$ or $(-3-1~2-4)$ or the pair $(-3-4)$ sits in the last two positions of $w$. If the pair $(-3-4)$ sits in the last two positions of $w$, then the argument  goes as in Case 2 above.\\
If we have the form $(-3-1~2-4)$, then we must add $n-4$ to $ \frac{(p-2)(q-2)}{2}$   to put 3 in the last position
and add $n-2$ to $ \frac{(p-2)(q-2)}{2}$ to put 4 in the last position. Then we must add 1 to split the signs
and 3 to sent $-3$ to its position in the original $w$, then the length of $w$ become $ \frac{(p-2)(q-2)}{2}+(n-4)+(n-2)+1+3=\frac{pq}{2}$ .\\
If we have the form $(-3-4-1~2)$, then we must add $n-4$ to $ \frac{(p-2)(q-2)}{2}$ to put 3 in the last position
and add $n-1$ to $ \frac{(p-2)(q-2)}{2}$ to put 4 in the position before the last position. Then we must add 1 to split the signs,  3 to send $-3$ to its position in the original $w$ and 1 to send $-4$ to its position in the origenal $w$. It follows that the length of $w$ is  $ \frac{(p-2)(q-2)}{2}+(n-4)+(n-3)+1+3+1=\frac{pq}{2}$.\\
\end{proof}
\begin{prop}
Every perfect harmonic permutation $w \in {\displaystyle S_{n}\ltimes \mathbb{Z}_2^{n}}$ has dimension $\frac{pq}{2}$.
\end{prop}
\begin{proof}
The argument goes exactly along the lines as that for the above Proposition.  Here it is in fact simpler, because only the pairs $(-1~2)$ and $(-3~4)$ appear.
\end{proof}
\section{Intersection points of Schubert duality}
 let $w\in W_I$ be a perfect harmonic permutation, in particular so that $S_w$ is of complementary dimension to the cycles. Recall that in this case either $S_w\cap C_\alpha =\emptyset $ or is pointwise $T_S$-fixed (see Theorem \ref{corcor} and corollary \ref{fixthm}). The main goal in this section is to compute all
 such intersection points. The argument in the case of $SO(p,q)$ is carried out by means of algorithms.  Nevertheless we are able to provide formulas for the cardinality of $S_w\cap C_\alpha $ when it is non-zero and the total number of cycles $C_\alpha $ for which this intersection is non-empty (see Theorem \ref{thm5.18}).
 \bigskip\noindent

 \begin {prop}\label{7.4.1}
If $w$ is a perfect harmonic permutation
so that $B_I.F_w$ intersects a cycle $C_\alpha $ at a point given by the $\varepsilon $-basis $(\varepsilon_1,\ldots ,\varepsilon_m)$ of $T_S$-eigenvectors, then for any such eigenvector $\varepsilon_k$ it follows that the $\varepsilon $-basis is given by
 $\varepsilon_k=e_{2j-1}+ie_{2j}$ or $e_{2j-1}-ie_{2j}$ if $q$ and $p$ are even , and  $\varepsilon_k=e_{2j-1}+ie_{2j}$ , $e_{2j-1}-ie_{2j}$ , $e_q$ , $e_{q+1}$ or $e_{2q+1}$ if $q$ or $p$ is odd, depending on the dimension $m$ and the signature $\alpha$.
 \end {prop}
 \begin {proof}
 This is a consequence of the following:
 \begin {enumerate}
\item $w$ is a perfect harmonic permutation and the flag basis is that of $w(F_I)$.
\item We have the following cases
\begin{eqnarray}
\bullet~ b.v_j &=& \eta_j(e_r +e_{2q-r+1})+\zeta_j(e_r-e_{2q-r+1})+\tilde{\eta}_j(e_{r+1} +e_{2q-r})+\tilde{\zeta}_j(e_{r+1}-e_{2q-r})+B_j \nonumber \\
&=& K_j+B_j,\nonumber.
 \end{eqnarray}
 
$\bullet ~b.v_j = \eta_j(e_{2q+r} +ie_{2q+r+1})+B_j=K_j+B_j,$\\

where $\tilde{\eta}_j= \pm i \eta_j$ and $\tilde{\zeta}= \pm i \zeta_j$ and $B_j$ does not involve the basis vectors $e_r $, $e_{r+1}$, $e_{2q-r+1}$ and $e_{2q-r}$.\\

$\bullet ~b.v_j=\eta_j(e_{2q+1})+B_j=K_j+B_j,~ \text{if}~m=2n+1$

 \item $\eta_j\not=0$ and $\tilde{\eta}_j \neq 0$.

 \item The intersection $S_w\cap C_\alpha $ is a flag defined by $T_S$-eigenvectors.
 \end {enumerate}
 From the expression for $v_j$ it is immediate that all of the possibilities in the statement
 occur.    Furthermore, since the  $\eta_j\not=0$ and $\tilde{\eta}_j \neq 0$ are non-zero, for every $j$ a non-zero contribution from $K_j$
 occurs in the sum
 $$
 v_j=\sum_{k\le j} c_{kj}b.v_{w_k}\,.
 $$
 Since $e_{2j-1}+ie_{2j}$ and $e_{2j-1}-ie_{2j}$ and $e_q$ or $e_{q+1}$ do not occur in $b.v_{w_k}$ for $k<j$, it follows that $\varepsilon_j=K_j+E_j$
 in the standard basis.  Finally, since $\varepsilon_j$ is a $T_S$-eigenvector, it follows that $\varepsilon_j=K_j$ and
 is of the type in the the statement of the proposition.
 \end {proof}

\noindent
As a result of the above Proposition the following corollary gives us all intersection points of $S_w \cap C_{\alpha}$:
\begin{cor}\label{cor12.4}
Let $D_{\alpha}$ be a flag domain parametrized by a sequence $\alpha$ and $w \in W_I$
 be a perfect harmonic permutation such that $S_w \cap D_{\alpha}\neq \phi$. Then the
following algorithm produces us all intersection points of $S_w \cap C_{\alpha}$:\\

\textbf{If $q$ is even:}
\begin{itemize}
\item Consider a copy of $\alpha$ denoted by $\beta$.
	\item  For any pair $(-(2j-1),2j), 1\leq j \leq \frac{q}{2},$ in $w$ if the corresponding signature of it in $\beta$ is $+-$ then replace this $+-$ in $\beta$ by $$<e_{2q-2j+1}+ie_{2q-2j+2},e_{2j-1}-ie_{2j}> \text{ or } <e_{2q-2j+1}-ie_{2q-2j+2},e_{2j-1}+ie_{2j}>,$$
	and if the corresponding signature of it in $\beta$ is $-+$ then replace this $-+$ in $\beta$ by
	$$<e_{2j-1}+ie_{2j},e_{2q-2j+1}-ie_{2q-2j+2}> \text{ or }  <e_{2j-1}-ie_{2j},e_{2q-2j+1}+ie_{2q-2j+2}>$$
	
	\item  For any pair $(-(2j-1),-2j), 1\leq j \leq \frac{q}{2},$ in $w$ if the corresponding signature of it in $\beta$ is $+-$ then replace this $+-$ in $\beta$ by $$<e_{2q-2j+1}+ie_{2q-2j+2},e_{2j-1}+ie_{2j}> \text{ or } <e_{2q-2j+1}-ie_{2q-2j+2},e_{2j-1}-ie_{2j}>,$$
	and if the corresponding signature of it in $\beta$ is $-+$ then replace this $-+$ in $\beta$ by
	$$<e_{2j-1}+ie_{2j},e_{2q-2j+1}+ie_{2q-2j+2}> \text{ or } <e_{2j-1}-ie_{2j},e_{2q-2j+1}-ie_{2q-2j+2}>$$
	
	\item For the remaining numbers, for each $q+1 \leq j \leq n-1$ replace the corresponding $+$ in $\beta$ by $(e_{2j-1}+i e_{2j})$ and for $\pm n$ replace the corresponding $+$ by $(e_{2n-1} \pm i e_{2n})$.
\end{itemize}
\textbf{If $q$ is odd:}
\begin{itemize}
	\item  For any pair $(-(2j-1),2j), 1\leq j \leq \frac{q-1}{2}$, in $w$ if the corresponding signature of it in $\beta$ is $+-$ then replace this $+-$ in $\beta$ by $$<e_{2q-2j+1}+ie_{2q-2j+2},e_{2j-1}-ie_{2j}> \text{ or } <e_{2q-2j+1}-ie_{2q-2j+2},e_{2j-1}+ie_{2j}>,$$
	and if the corresponding signature of it in $\beta$ is $-+$ then replace this $-+$ in $\beta$ by
	$$<e_{2j-1}+ie_{2j},e_{2q-2j+1}-ie_{2q-2j+2}> \text{ or } <e_{2j-1}-ie_{2j},e_{2q-2j+1}+ie_{2q-2j+2}>$$
	\item  For the number $-q$ replace the corresponding $+$ in $\beta$ by $e_{q+1}$.
	\item For the remaining numbers, for each $q+1 \leq j \leq n-1$, replace the corresponding $+$ in $\beta$ by $(e_{2j-1}+i e_{2j})$ and for $\pm n$ replace the corresponding $+$ by $(e_{2n-1} \pm i e_{2n})$.
\end{itemize}
Each point obtained from this algorithm is a point of the intersection of $S_w \cap D_{\alpha}$ .
\end{cor}

\begin{thm}\label{thm5.18}
 A Schubert variety $S_w$ which is parametrized by a perfect harmonic permutation $w$ intersects $2^{\frac{q}{2}}$ flag domains if $q$ is even and intersects the base cycles of these flag domains in $2^q$ points. If $q$ is odd, it intersects $2^{\frac{q-1}{2}}$ flag domains and intersects the base cycles of these flag domains in $2^{q-1}$ points.
\end{thm}
\begin{proof}
Let $w \in W_I$ be a perfect harmonic permutation. We first show that if $q$ is even, then $S_w\cap C_\alpha $ consists of $2^q$ points and $2^{q-1}$ points if q is odd. Since $w$ is a perfect harmonic permutation, we have two cases.  In the first case if $w$ contains the
pair $(-1~2)$ then the pair $(-1~2)$ sits inside consecutive boxes in $w$. The goal here is to show that there is exactly $4$-possibilities for this pair which can be
completed to maximal isotropic flag. We begin by considering the $B_I$-orbit of $<e_1-e_{2q}>$. By Remark \ref{rem5.4} the elements  in this orbit have the form
$$\beta_1(e_1+e_{2q})+....+\beta_{2n}(e_1-e_{2q}).$$
The question is how many 1-dimensional subspaces (spanned by vectors in Proposition \ref{7.4.1}) do we have such that these subspaces can be complete to
maximal isotropic flag? To compute this number we denote by $v_1$ the vector we have from the first step which spans the 1-dimensional subspace, in the second step we consider the orbit $B_I.<e_2+e_{2q-1}>$, which has points of the form
$$\beta_1(e_1+e_{2q})+\beta_{2}(e_2+e_{2q-1}).$$
The $2$-dimension subspace in the flag is spanned by linear combinations of the form
$$v_2=\beta_1(e_1+e_{2q})+\beta_{2}(e_2+e_{2q-1})+\gamma v_1.$$
 Note that $v_2$ should be in $E^+$ or $E^-$ and of the form stated in proposition \ref{7.4.1}, and therefore $v_1$ should contain the terms $e_1$ and $e_2$ or the terms $e_{2q-1}$ and $e_{2q}$. Thus we have the following 4 possibilities of $v_1$ as follows: $e_1-ie_2,e_1+ie_2,e_{2q-1}+ie_{2q},e_{2q-1}-ie_{2q}$.\\
If $v_1=e_1\mp ie_2$, then for $\beta_1= \pm i,~\beta_2=1,$ and $~\gamma=\mp i$, the vector $v_2$ is $v_2=e_{2q-1} \pm ie_{2q}$, so the 2-dimensional subspace corresponding to the pair $(-1~2)$ is spanned by $$<e_1\mp ie_2,e_{2q-1} \pm ie_{2q}>.$$
If $v_1=e_{2q-1} \mp ie_{2q}$, then for $\beta_1=1,~\beta_2=\pm i,$ and $~\gamma=\mp i$, it follows that $v_2$ is $v_2=e_1\pm ie_2$. Thus the 2-dimensional subspace corresponding to the pair $(-1~2)$ is spanned by $$<e_{2q-1}\mp ie_{2q},e_1\pm ie_2,>.$$
Having constructed the $2$-dimensional space we ignore the pair $(-12)$ in $w$ and repeat the same steps for the next pairs step by step, So if $w$ contains the pair $(-(2j-1),2j), 1\leq j \leq \frac{q}{2}$,
then in the same way the only possible 2-dimensional subspaces which can be completed to maximal isotropic flags are the subspaces spanned by: $$<e_{2j-1}\mp ie_{2j},e_{2q-2j+1}\pm ie_{2q-2j+2}> \text{ or }  <e_{2q-2j+1}\pm ie_{2q-2j+2},e_{2j-1}\mp ie_{2j}>.$$

\noindent

The second case is that where $w$ contains the
pair $(-1-2)$ which sits inside consecutive boxes in $w$. look at the orbit $B_I.<e_1-e_{2q}>$, then this orbit has points of the form $$\beta_1(e_1+e_{2q})+....+\beta_{2n-1}(e_2-e_{2q-1})+\beta_{2n}(e_1-e_{2q}).$$
 We also consider the orbit $B_I.<e_2-e_{2q-1}>$ which has points of the form $$\beta_1(e_1+e_{2q})+....+\beta_{2n-1}(e_2-e_{2q-1}).$$
Then, if we have 1-dimensional subspace spanned by $v_1$ from the first step, the 2-dimensional subspace in our flag spanned by $v_1$ and a vector $v_2$ of the form
$v_2=\beta_1(e_1+e_{2q})+\beta_{2}(e_2+e_{2q-1})+\beta_{3}(e_2-e_{2q-1})+\gamma v_1$. So we have 4-possibilities of $v_1$ which can be extend to maximal isotropic flag. These are : $e_1-ie_2,e_1+ie_2,e_{2q-1}+ie_{2q},e_{2q-1}-ie_{2q}$.\\
If $v_1=e_1\mp ie_2$, then for $\beta_1=\mp i,\beta_2=0,\beta_3=-1,$ and $\gamma=\pm i$, the vector $v_2$ is $v_2=e_{2q-1}\mp ie_{2q}$. Thus the 2-dimensional subspace corresponding to the pair $(-1-2)$ is spanned by $$<e_1\mp ie_2,e_{2q-1}\mp ie_{2q}>.$$
If $v_1=e_{2q-1}\mp ie_{2q}$, then for $\beta_1=1,\beta_2=0,\beta_3=\mp i,$ and $\gamma=\mp i$, the vector $v_2$ is $v_2=e_1\pm ie_2$. As a result the 2-dimensional subspace corresponding to the pair $(-1-2)$ is spanned by $$<e_{2q-1}\mp ie_{2q},e_1\mp ie_2,>.$$
Having determined the $2$-dimensional subspace, we ignore the pair $(-1-2)$ from $w$ and repeat the same steps for the next pairs step by step. More generally if $w$ contains the pair $(-(2j-1),-2j), 1\leq j \leq \frac{q}{2}$,
then by the same method the only possible 2-dimensional subspaces which can be completed to maximal isotropic flag are $$<e_{2j-1}\mp ie_{2j},e_{2q-2j+1}\mp ie_{2q-2j+2}> \text{ or } <e_{2q-2j+1}\mp ie_{2q-2j+2},e_{2j-1}\mp ie_{2j}>$$
Therefore for each pair of $w$ we have 4-possibilities.

\bigskip\noindent
For the remaining numbers, recall that $w$ is a perfect harmonic permutation. In particular all numbers in the remaining boxes sit in an increasing order. Thus the only possible for these which can be completed to a maximal isotropic flag are the following point:
If $n$ is positive, the point is the flag associated to the ordered basis $$e_{2q+1}+ie_{2q+2},...,e_{2n-1}+ie_{2n},$$ and if $n$ is negative, then the point is the flag associated to the ordered basis $$e_{2q+1}+ie_{2q+2},...,e_{2n-3}+ie_{2n-2},e_{2n-1}-ie_{2n}.$$
Hence, if $q$ is even, then $w$ contains $\frac{q}{2}$ pairs and each pair has 4-possibilities. Therefore in this case the number of possible intersection points is $4^{\frac{q}{2}}.1=2^q$. If $q$ is odd, then $w$ contains $\frac{q-1}{2}$ pairs and each pair has 4-possibilities then the number of possible intersection points is $4^{\frac{q-1}{2}}.1=2^{q-1}$.

\bigskip\noindent
Finally, we show that the points described above belong to exactly $\frac{q}{2}$ flag domains if q is even and to $\frac{q-1}{2}$ flag domains if q is odd. For this recall that for each
pair in $w$ we have 4-possibilities of 2-dimensional subspaces and note that the signature of these subspaces are $+-$ and $-+$, and the signature of the remainder is
$++...++$. Thus the number of flag domains which have non-empty intersection with $S_w$ for the given $w$  is $2^{\frac{q}{2}}$ flag domains if q is even and is $2^{\frac{q-1}{2}}$ flag domains if q is odd.

\bigskip\noindent

Also, since each two of these 4-possibilities of 2-dimensional subspaces have the same signature then for any fixed signature for a flag domain there is
$2^{\frac{q}{2}}$ points belong to the base cycle of that flag domain if q is even and  $2^{\frac{q-1}{2}}$ points belong to the same base cycle of that flag domain if q is odd.
\end{proof}
\bigskip\noindent
\begin{rem}
In the case of the group $SL_{2n}^C$ with the real form $SU(p, q)$, Brecan \cite{Bre} shows that the number of Iwasawa-Schubert
varieties which intersect at least one base cycle and has the minimal dimension $pq$ is $(2n-1).(2n-3)...(2n-2q+1)$. But for the group $SO(2n,\mathbb{C})$ with real form $SO(p,q)$
, if q is even, the number of Schubert varieties $S_w$ which have the minimal dimension $\frac{pq}{2}$ and intersect at least one base cycle  is $n.(n-2)...(n-q+2)$, and
if q is odd, the number of Schubert varieties $S_w$ which have the minimal dimension $\frac{pq-1}{2}$ and intersect at least one base cycle  is $(n-2).(n-4)...(n-q+1)$.\\
In the case of $SO(2n+1,\mathbb{C})$ with real form $SO(p,q)$
, if q is even, the number of Schubert varieties $S_w$ which have the minimal dimension $\frac{pq}{2}$ and intersect at least one base cycle   is $(n-1).(n-3)...(n-q+1)$, and
if q is odd, the number of Schubert varieties $S_w$ which has the minimal dimension $\frac{pq}{2}$ and intersect at least one base cycle   is $(n-2).(n-4)...(n-q+1)$.\qed
\end{rem}

\bigskip\noindent
The following remark, which is a consequence of the proof of Theorem \ref{thm5.18}, describes all intersection points of $S_w$ with the base cycles $C_{\alpha}$.

\begin{rem}
To determine all intersection points between the base cycles and the Iwasawa Schubert variety $S_w$ of complimentary dimension we will define a set for each case of $q$:\\
\textbf{If $q$ is even:} Let $ w \in W_I$ be a perfect harmonic permutation and define \\
$Swite_w :=\{\psi(w_r): w_r$ ~is~obtained~from~$w$~by~switching ~none,~some or~all~pairs\\
 $(-(2i-1),2i)~by ~(2i,-(2i-1))~or~(-2i,2i-1)~or~((2i-1),-2i)$~and~switching ~none,~some\\ or~all~pairs~$(-(2i-1),-2i)~$ by $~(-2i,-(2i-1))~or~(2i,(2i-1))~or~((2i-1),2i), 1\leq i\leq \frac{q}{2}\} \subset W_S$.
 Define $\mathbb{F}_e(Fix~T_S)$ to be the set all maximally $b$-isotropic flags associated to the basis
$$e_1+ie_2,e_3+ie_4,......,e_{2n-1}+ie_{2n},e_{2n-1}-ie_{2n},......,e_1-ie_2$$
Let $M_w\subset \mathbb{F}_e(Fix~T_S)$ be the set of all maximally $b$-isotropic flags associated to all elements in $Swite_w $. Note that we have $\frac{q}{2}$ of
the pairs $(-(2i-1),2i)$ and $(-(2i-1),-2i),1\leq i\leq \frac{q}{2} $
 in any $w \in W_I$, and for each pair we have 4 possibles to switch it, so the cardinality of $Swite_w $ is $4^{\frac{q}{2}}=2^q$.
 The set $Swite_w $ gives us all intersection points of $S_w$ and each $2^{\frac{q}{2}}$ of these points belong to only one flag domain where these points of intersection sits in the base cycle of that flag domain. \\
 \textbf{If $q$ is odd:}  Let $ w \in W_I$ be a perfect harmonic permutation and define\\
 $Swito_w :=\{\psi(w_r): w_r$ ~is~obtained~from~$w$~by~switching ~none,~some or~all~pairs \\
 $(-(2i-1),2i)~by ~(2i,-(2i-1))~or~(-2i,2i-1)~or~((2i-1),-2i)$, $1\leq i\leq \frac{q}{2}\} \subset W_S$.
 Define $\mathbb{F}_o(Fix~T_S)$to be the set all maximal b-isotropic flags associated to the basis
$$e_1+ie_2,e_3+ie_4,...,e_{q-2}+ie_{q-1},e_{q+2}+ie_{q+3},...,e_{2n-1}+ie_{2n},e_q,$$ $$e_{q+1},e_{2n-1}-ie_{2n},...,e_{q+2}-ie_{q+3},e_{q-2}-ie_{q-1},...,e_1-ie_2$$
Let $M_w\subset \mathbb{F}_o(Fix~T_S)$ be the set of all maximally $b$-isotropic flags associated to all elements in $Swito_w $. Note that we have $\frac{q-1}{2}$ of the
pairs $(-(2i-1),2i), 1\leq i\leq \frac{q}{2} $ in any $w \in W_I$, and for each pair we have 4 possibilities for switching it. Hence the cardinality of $Swito_w $ is $4^{\frac{q-1}{2}}=2^{q-1}$.
  The set $Swito_w $ gives us all intersection points of $S_w$ and each $2^{\frac{q-1}{2}}$ of these points belong to only one flag domain where these points of intersection sits in the base cycle of that flag domain.
\end{rem}

\begin{ex}
 In $G_0=SO(6,4)$, fix $w=(-35-142)$ a perfect harmonic permutation, then $Swite_w=\{(251-3-4),(25-1-34),(-2513-4),(-25-134),(-3512-4),\\(351-2-4),(-35-124),
 (35-1-24),(25-4-31),(-25-431),(254-3-1),(-2543-1),\\(-35-421),(35-4-21),(-3542-1),(354-2-1)\} $
\end{ex}

 \begin{ex}
   In $G_0=SO(5,3)$, fix $w=(-124-3)$ a perfect harmonic permutation, then $Swito_w=\{(1-23-4),(-213-4),(2-13-4),(-123-4)\} $.
 \end{ex}

\noindent
Recall that in the cases of $SP(2n,\mathbb{R})$ and $SO^*(2n)$ every flag domain intersects all Schubert varieties of complementary dimension.
But in the case of $SO(p,q)$ we don't have this property except in a very special case. We explain this case in the following example.
\begin{ex}\label{7.4.11}
If $n=q+1$, then the flag domain parametrized by the sequence $$\alpha=+-+-...+-+-+$$ intersects all Schubert varieties of dimension $\frac{pq}{2}$ if $q$ is even.
And the flag domain parametrized by the sequence $$\beta=+-+-...+-+-++$$ intersects all Schubert varieties of dimension $\frac{pq-1}{2}$ if $q$ is odd.
\end{ex}


\chapter{Cycle intersection for $SP(2p,2q)$}

Our work here is devoted to the case of the real form $SP(2p,2q)$ of $SP(2n,\mathbb{C})$. As in the case $SO(p,q)$, the results here are stated in terms of algorithms (See Definitions \ref{majorD1} and \ref{majorD2}).

\section{Conditions for $S_w \cap C_{\alpha} \neq \emptyset$}

\begin{defn}\label{majorD1}
An element $w \in W$ is called \textit{\textbf{a major permutation}} if it satisfies the following conditions:\\
In the one line notation of the permutation the number $-(2i-1), 1\leq i \leq q$, sits in any place to the left of the number $(2i)$ or $(-2i), 1\leq i \leq q$, or the number $(-2i), 1\leq i \leq q$, sits in any place to the left of the number $(2i-1)$ or $-(2i-1), 1\leq i \leq q$,  and the order of the numbers $2q+i$ or $-(2q+i)$, where $1\leq i \leq p-q$ is arbitrary.
\end{defn}

\begin{ex}
In $SP(4,2)$ the permutation $(-1 25-34)$ is a major permutation while $(53 1-2-4)$ is not.
\end{ex}

		\bigskip\noindent
The standard basis and the basis (\ref{basis4.2}) define two maximal tori for $SP(2n,\mathbb{C})$,and therefore they define two isomorphic Weyl groups.
Define $W_{T}$ to be the Weyl group with respect to  the standard basis $\{e_1,...,e_{2n}\}$ of $\mathbb{C}^{2n}$.
Let $W_I$ be the Weyl group with respect to the basis
\begin{eqnarray}\label{basis4.2}
 e_1+e_{2q}, e_{2n-2q+1}+e_{2n},...,e_q+e_{q+1},e_{2n-q}+e_{2n-q+1},e_{2q+1},e_{2q+2},...,e_{n},e_{n+1},&&  \nonumber \\
 e_{n+2},...,e_{2n-q},e_q-e_{q+1},e_{2n-q}-e_{2n-q+1},...,e_1-e_{2q},e_{2n-2q+1}-e_{2n}.~~~~~~~~~& &
\end{eqnarray}

\noindent
 Each of these Weyl groups isomorphic to $S_n \ltimes \mathbb{Z}_2^{n}$. For later use define the bijective map $\psi$ between $W_I$ and $W_{S}$ by
 $\psi( \pm(2i-1))=\pm (2q-i+1),~ \psi(\pm 2i)=\mp i$ if $1\leq i\leq  q$ and $\psi(\pm i)= \pm i $ if $i>2q$. Note that $-(2q-i+1)=2n-2q-i+2$.

		\bigskip\noindent

\begin{rem}\label{rem6.4}
For $v_i$ any basis vector of the type in \ref{basis4.2} and $b\in B_I$ the form of $b(v_i)$ is given as follows:
\begin{itemize}
	\item $b.(e_i - e_{2n-i+1})=\eta_i(e_i +e_{2q-i+1})+\zeta_i(e_i-e_{2q-i+1})+B_i=K_i+B_i,$
\item $b.(e_i+ e_{2n-i+1})=\lambda_i(e_i +e_{2q-i+1})+A_i,$
\item $b(e_{2n-q+i} - e_{2n-i+1})=\beta_i(e_{2n-q+i} + e_{2n-i+1})+\delta_i(e_{2n-q+i} - e_{2n-i+1})+\tilde{B}_i=\tilde{K}_i+\tilde{B}_i,$
\item $b(e_{2n-q+i} + e_{2n-i+1})=\mu_i(e_{2n-q+i} + e_{2n-i+1})+\tilde{A}_i,$
\item $b(e_{2q+j})=\nu_j(e_{2q+j})+\tilde{A}_i,$
\end{itemize}
 with $\lambda_i\not=0$,$\mu_i\not=0$ and $\nu_j\not=0$.\\
Note that in the above orbits, if $b$ is chosen appropriately, then we can arrange the linear combinations above to have the vectors in the standard basis.\qed
 \end{rem}


\begin{prop}\label{lemmajor}
If $w$ is a major permutation, then the flag $F_{\psi(w)}$ belongs to the orbit $B_I(F_w)$.
\end{prop}
\begin{proof}
let $w \in W_I$ be a major permutation. Denote by
$$\{0\} \subset <u_{w_1}> \subset <u_{w_1},u_{w_2}> \subset ...... \subset <u_{w_1},...,u_{w_n}>$$
the first  $n$ subspaces of $F_w=w(F_I)$, where $u_{w_i}$ is a vector from the basis (\ref{basis4.2}).
Let $\tilde{w} \in W_{T} $ be the image of $w$ under the bijection $\psi$ and let $Y_{\tilde{w}}$ be the isotropic flag associated to $\tilde{w}$ such that the first half of $Y_{\tilde{w}}$ is
$$\{0\} \subset <e_{\tilde{w}_1}> \subset <e_{\tilde{w}_1},e_{\tilde{w}_2}> \subset ...... \subset <e_{\tilde{w}_1},...,e_{\tilde{w}_n}>$$
where $\varepsilon_i=e_{\tilde{w}_i}$ is a vector in the standard basis of $\mathbb{C}^{2n}$. To show that the flag $Y_{\tilde{w}}$ is an intersection point in $B_I.F_w\cap C_\alpha $ we will construct a $b$ with $b(F_w)=Y_{\tilde w}$ by using induction. From the definition of a major permutation, there are three possibilities for $w_1$: $w_1=-(2i-1)$ ,$w_1=-2i$, $1\leq i\leq q$, or $|w_1|>2q$\\
 \textbf{Case 1:} ~If $|w_1|>2q $, then $u_{w_1}=e_{\tilde{w}_1}$. Thus the orbit $B_I.<u_{w_1}>$ contains the point $<e_{\tilde{w}_1}>$.\\
 \textbf{Case 2:} ~If $w_1=-(2i-1), 1\leq i\leq q$, we consider the orbit $B_I.<e_{2n-2q-i+2}-e_{2n-i+1}>=B_I.<e_{\tilde{w}_1}-e_{2n-i+1}>$.
 By Remark (\ref{rem6.4}) above, the orbit $B_I.<e_{\tilde{w}_1}-e_{-i}>$ contains a point of the form $$y=<\alpha_1 (e_{\tilde{w}_1}-e_{2n-i+1})+ \alpha_2 (e_{\tilde{w}_1}+e_{2n-i+1})>$$
 where $\alpha_1=\pm \alpha_2$. By taking $\alpha_1=\alpha_2=\frac{1}{2}$, it follows that $y=<e_{\tilde{w}(1)}>$.\\
 \textbf{Case 3:} ~If $w_1=-2i$, then we consider the orbit $B_I.<e_{\tilde{w}_1}-e_{2q-i+1}>$.
 By Remark (\ref{rem6.4}) above, the orbit $B_I.<e_{\tilde{w}_1}-e_{2q-i+1}>$ contains a point of the form $$y=<\alpha_1 (e_{\tilde{w}_1}-e_{2q-i+1})+ \alpha_2 (e_{\tilde{w}_1}+e_{2q-i+1})>$$
 where $\alpha_1=\pm \alpha_2$. By taking $\alpha_1=\alpha_2=\frac{1}{2}$, it follows that $y=<v_{\tilde{w}_1}>$.\\
Now apply induction to build the full isotropic flag. Assume that we built the first $j$-vectors of $b.v_{w_j}$ where $j<n$, which yield the first $j$ isotropic subspaces of the flag,
$$\{0\} \subset <e_{\tilde{w}_1}> \subset <e_{\tilde{w}_1},e_{\tilde{w}_2}> \subset ...... \subset <e_{\tilde{w}_1},...,e_{\tilde{w}_j}>$$
To construct the $j+1$-vector of $b.v_{w_{j+1}}$ to obtain the subspace $V_{\tilde{w}_{j+1}}$ we must consider $5$-cases:\\
\textbf{Case 1:} If $|w_{j+1}|>q$, then $w_{j+1}=\tilde{w}_{j+1}$ and $v_{\tilde{w}_{j+1}}=u_{w_{j+1}}$, and consequently the orbit $B_I.<u_{w_{j+1}}>$ contains the point $e_{\tilde{w}_{j+1}}$.\\
\textbf{Case 2:} If $w_{j+1}=-(2i-1)$, then we consider the orbit $B_I.<e_{\tilde{w}_{j+1}}-e_{2n-i+1}>$.
 By using Remark (\ref{rem6.4}) we see that this orbit contains the points of the form $$y=<\alpha_1 e_{\tilde{w}_{j+1}}-e_{2n-i+1}+\alpha_2 e_{\tilde{w}_{j+1}}-e_{2n-i+1})>$$
 where $\alpha_1=\pm \alpha_2$.  By taking $\alpha_1=\alpha_2=\frac{1}{2}$, it follows that $y=<e_{\tilde{w}_{j+1}}>$. \\
  \textbf{Case 3:} If $w_{j+1}=-2i$ , then the orbit is $B_I.<e_{\tilde{w}_{j+1}}-e_{2q-i+1}>$ . In this case the point
$$y=<\alpha_1 (e_{\tilde{w}_{j+1}}-e_{2q-i+1}) +\alpha_2 (e_{\tilde{w}_{j}}-e_{2q-i+1})>$$
belongs to the orbit $B_I.<e_{\tilde{w}_{j+1}}-e_{2q-j+1}>$, and for $\alpha_1=\alpha_2$  the point is $y=<e_{\tilde{w}_{j+1}}>$.\\
\textbf{Case 4:} If $w_{j+1}=2i$, then we consider the orbit $B_I.<e_{\tilde{w}_{j+1}}+e_{2n-i+1}>$.
Since $w$ is major permutation, then the vector $<e_{\tilde{w}_{j+1}}>$ is sitting before the vector $<e_{\tilde{w}_{j}}+e_{2n-i+1}>$ in the flag. Thus this orbit contains a point
of the form $$y=<\alpha_1 e_{\tilde{w}_{j}}+\alpha_2 e_{\tilde{w}_{j+1}+e_{2n-i+1}}>$$
 where $\alpha_1=\pm \alpha_2$.  By taking $\alpha_2=-\alpha_1$, it follows that $y=<e_{2n-i+1}>=<e_{\tilde{w}_{j+1}}>$. \\
  \textbf{Case 5:} If $w_{j+1}=2i-1$ , then the orbit is $B_I.<e_{\tilde{w}_{j}}-e_{2q-j+1}>$.
Since $w$ is a major permutation, the vector $<e_{\tilde{w}_{j+1}}>$ sits before the vector $<e_{\tilde{w}_{j}}+e_{2q-i+1}>$ in the flag. Thus this orbit contains a point
of the form $$y=<\alpha_1 e_{\tilde{w}_{j}}+\alpha_2 e_{\tilde{w}_{j+1}}+e_{2q-i+1}>$$
 where $\alpha_1=\pm \alpha_2$.  By taking $\alpha_2=-\alpha_1$, it follows that $y=<e_{2q-i+1}>=<e_{\tilde{w}_{j+1}}>$.
 Therefore, the vector $b.v_{w_{j+1}}$ is constructed to obtain the flag
  $$\{0\} \subset <e_{\tilde{w}_1}> \subset <e_{\tilde{w}_1},e_{\tilde{w}_2}> \subset ...... \subset <e_{\tilde{w}_1},...,e_{\tilde{w}_{j+1}}>.$$ \\
	Thus by induction we observe that $b\in B_I$ can be constructed with $b(F_w)=Y_{\tilde{w}}$.\\
\end{proof}
\begin{thm}\textbf{(Major Permutation Theorem)}.
The following are equivalent
\begin{description}
	\item[~~~~~~~(i)]  $w$ is Major.
  \item[~~~~~~~(ii)]  $B_I(F_w)\cap D_\alpha \not=\emptyset$ for some $\alpha $.
		\end{description}
Under either of these conditions, for every $\alpha $ the intersection $B_I(F_w)\cap C_\alpha$
 contains a $T_s$ fixed point.
\end{thm}
\begin{proof}
\begin{description}
	\item[(i)$\Rightarrow$ (ii)] Comes directly from Proposition \ref{lemmajor}.
	\item[(ii)$\Rightarrow$ (i)] The proof of this part is essentially the same as the proof of the second part of Theorem \ref{thmSOPQ}.\\
\end{description}
\end{proof}

\section{Introduction to the combinatorics}
 For the remainder of this chapter we only discuss the intersection properties of the Iwasawa-Schubert cells which are of complementary dimension to $C_0$.
Consider the standard  basis $\{e_1,...,e_{2n}\}$ of $\mathbb{C}^{2n}$  where $E^-:=<e_1,...,e_{q},e_{2n-q+1},...,e_{2n}>$ and $E^+:=<e_{q+1},...,e_{2n-q}>$.
Since the the maximal compact subgroup of $SP(2p,2q)$ is $\tilde{K}_0 := SP(2q) \times SP(2p)$ corresponding to the decomposition $\mathbb{C}^{2n}=E^-\oplus E^+$, it follows that for a fixed flag domain $D_{a,b}$, the base cycle $C_0$ is the set
$$C_0 =\{F \in Z : \text{dim}(V_i \cap E^-)=\sum_{j=1}^i a_j~ and~ \text{dim}( V_i \cap E^+) =\sum_{j=1}^i b_j , 1 \leq i \leq 2n\}.$$
 Since $C_0=\frac{K}{K \cap B_{z_0}}$ where $z_0$ is a point in a flag domain , the dimension of the base cycle $C_0$ is ~$\text{dim}~C_0=p^2+q^2$,
 and the dimension of the Iwasawa Schubert variety $S_w$ is given by $\text{dim}~S_w=2pq$.

\begin{defn}\label{majorD2}
A major permutation $w \in S_n \ltimes \mathbb{Z}_2^{n-1}$ is called \textit{\textbf{perfect major permutation}} if it is constructed by the following algorithm:

 \begin{enumerate}
\item Start with a sequence of $n$ empty boxes which are to be filled in order to construct $w$.
\item Consider the pairs $(-2j,(2j-1))$ for all $1\leq j \leq q$, i.e., the pairs are $(-2,1)$,$(-4,3)$,...,\\$(-2q,2q-1)$.
 \item Step by step, starting from 1 until $q$, for each $1\leq j \leq q$, place the pair $(-2j,(2j-1))$ in any box in $w$ such that the components of this pair sits as
 close as possible to each other.
 \item After all pairs are placed, the numbers $(2q+i)$ for $1\leq i \leq p-q$, are placed in the remaining spots in strictly increasing order.
\end{enumerate}
\end{defn}

\begin{ex}
If $p=3,~q=2$, then the perfect major permutations are:$(-2~1)(-4~3)5,$\\$(-2~1)5(-4~3),5(-2~1)(-4~3),(-4~3)(-2~1)5,(-4~3)5(-2~1),5(-4~3)(-2~1),
(-3-2~1~4)5,\\5(-3-2~1~4)$.
\end{ex}

\begin{ex}
If $p=10,~q=6$, then the element $78(-6~5)(-4~3)(-2~1)5$ is a perfect major permutation while the element $78(-3-4)(-2~1)(-6~5)$ is not perfect major permutation.
\end{ex}

\bigskip\noindent
\begin{prop}
Every perfect major permutation $w \in {\displaystyle S_{n}\ltimes \mathbb{Z}_2^{n}}$ has length $2pq$ .
\end{prop}
\begin{proof}
The proof of this Proposition is essentially the same as the proof of lemma \ref{lem5.13}.\\
If $w \in {\displaystyle S_{n}\ltimes \mathbb{Z}_2^{n}}$ is a perfect major permutation,carried out by i pair $(-2~1)$ sits inside consecutive boxes in $w$. Our proof here is carried out by induction on the dimension of the flag manifold.
First, remove the pair $(-2~1)$ from $w$ to have a new permutation $v$ consisting of the numbers $\{3,4,...,n\}$. Define a function $f:\{3,4,...,n\}\longrightarrow \{1,...,n-2\}$
by $f(i)=i-2$. This is a bijective map which sends $v$ to $\hat{w} \in {\displaystyle S_{n-2}\ltimes \mathbb{Z}_2^{n-2}} $. But $\hat{w}$ is perfect major permutation
in $ {\displaystyle S_{n-2}\ltimes \mathbb{Z}_2^{n-2}} $ and consequently
\begin{eqnarray}\label{dim1}
 l(\hat{w})&=& 2(p-1)(q-1).
\end{eqnarray}
 Since $v=f^{-1}(\hat{w})$ it follows that  $v$ has the same length as $\hat{w}$. For the second step put the numbers $12$ to the left of $v$ to obtain an element
$\tilde{w} =(12)v \in {\displaystyle S_{n}\ltimes \mathbb{Z}_2^{n}}$ with length $l(\hat{w})=2(p-1)(q-1)$. To split the sign of $2$ we must add $n-1$ to the length in (\ref{dim1}). Hence we have the element $(1v-2)$ with length $ l(\hat{w})+n-1$. Now return to
the original $w$ and define $h$ to be the number of positions to the left of $(-2)$ and $n-h-2$ to  be the number of positions to the right of $1$. Recall that in $w$ the pair $(-2~1)$ sits inside consecutive boxes, and therefore $-2$ must cross $n-h-2+1$ positions to stay in
the last position. Furthermore,  $1$ must cross $h$ positions to stay in the first position. Consequently the length of $w$ is
$$l(w) = 2(p-1)(q-1)+n-1+h+n-h-2+1=2pq\,.$$

\end{proof}

\section{Intersection points of Schubert duality}
  Now we turn to the final step for the case  $SP(2p,2q)$. Let $w\in W_I$ be a perfect major permutation and recall that if $S_w\cap C_\alpha \not=\emptyset $, then, since $S_w$ is of complementary dimension to $C_\alpha $, it follows that the intersection is just isolated points which are $T_S$-fixed points. The main goal  in the present paragraph is to compute all such intersection points. As in the case of $SO(p,q)$ the argument in the case of $SP(2p,2q)$ carried out via algorithms.

 \bigskip\noindent

 \begin {prop}\label{8.4.1}
If $w$ is a perfect major permutation
such that $B_I.F_w$ intersects a cycle $C_\alpha $ at a point given by an $\varepsilon $-basis $(\varepsilon_1,\ldots ,\varepsilon_m)$ of $T_S$-eigenvectors, then for any such eigenvector $\varepsilon_k$ it follows that $\varepsilon_k=e_{q-i+1} \text{~or}~e_{2n-q+i},~1\leq i\leq q$, or $\varepsilon_k=e_{q+i},\\1\leq i\leq 2p$, depending on the dimension $m$ and the signature $\alpha$.
 \end {prop}
 \begin {proof}
 This is a consequence of the following:
 \begin {enumerate}
\item $w$ is a perfect major permutation and the flag basis is that of $w(F_I)$.

\item The orbit $b.v_{w_i}$ has the following possible forms:\\
 $\bullet~b.v_{w_i}=b.(e_i - e_{2n-i+1})=\eta_i(e_i +e_{2q-i+1})+\zeta_i(e_i-e_{2q-i+1})+B_i=K_i+B_i,$\\
$\bullet ~b.v_{w_i}=b.(e_i+ e_{2n-i+1})=\lambda_i(e_i +e_{2q-i+1})+A_i,$\\
$ \bullet ~b.v_{w_i}=b(e_{2n-q+i} - e_{2n-i+1})=\beta_i(e_{2n-q+i} + e_{2n-i+1})+\delta_i(e_{2n-q+i} - e_{2n-i+1})+\tilde{B}_i=\tilde{K}_i+\tilde{B}_i,$\\
$\bullet ~b.v_{w_i}=b(e_{2n-q+i} + e_{2n-i+1})=\mu_i(e_{2n-q+i} + e_{2n-i+1})+\tilde{A}_i,$\\
or\\
$\bullet ~b.v_{w_i}=b(e_{2q+j})=\nu_j(e_{2q+j})+\tilde{A}_i.$\\

 \item $\lambda_i\not=0$,$\mu_i\not=0$,$\nu_j\not=0$,$\eta_j\not=0$, and $\beta_i \neq 0$.

 \item The intersection $S_w\cap C_\alpha $ is a flag defined by $T_S$-eigenvectors.
 \end {enumerate}
 From the expression for $v_{w_i}$ it is immediate that all of the possibilities in the statement
 occur.    Furthermore, since the  $\lambda_i,\mu_i,\nu_j,\eta_j$, and $\beta_i $ are non-zero, for every $i$ a non-zero contribution from $K_i$
 occurs in the sum
 $$
 v_i=\sum_{k\le i} c_{ik}b.v_{w_k}\,.
 $$
 Since for each case in item (2) above, the vectors $e_i , e_{2n-i+1}$ do not occur in $b.v_{w_k}$ for $k<j$ in the first and second lines in item (2), and the vectors $ e_{2n-q+i} , e_{2n-i+1}$  do not occur in $b.v_{w_k}$ for $k<j$ in the third and fourth lines in item (2), and the vector $e_{2q+j}$ do not occur in $b.v_{w_k}$ for $k<j$ in the last line in item (2) above. It follows that $\varepsilon_i=K_i+E_i$ in the standard basis.  Finally, since $\varepsilon_i$ is a $T_S$-eigenvector, it follows that $\varepsilon_i=K_i$ and
 is of the type in the the statement of the proposition.
 \end {proof}

As a result of the above Proposition the following algorithm gives us all intersection points of $S_w \cap C_{\alpha}$:
\begin{cor}\label{cor6.17}
Let $D_{\alpha}$ be a flag domain parametrized by a sequence $\alpha$ and $w \in W_I$
 be a perfect major permutation such that $S_w \cap D_{\alpha}\neq \phi$. Then the
following algorithm produces us all intersection points of $S_w \cap C_{\alpha}$:
\begin{itemize}
\item Consider a copy of $\alpha$ denoted by $\beta$.
	\item  For any pair $(-2j,(2j-1)), 1\leq j \leq q$, in $w$ if the corresponding signature of it in $\beta$ is $+-$, then replace the pair $+-$ in $\beta$ by $$<e_{2q-j+1},e_j> \text{~or~} <e_{2n-2q+j},e_{2n-j+1}>,$$
	and if the corresponding signature of it in $\beta$ is $-+$, then replace the pair $-+$ in $\beta$ by $$<e_j,e_{2q-j+1}> \text{~or~} <e_{2n-j+1},e_{2n-2q+j}>.$$
	\item For the remaining numbers, for each $2q+1 \leq j \leq n-1$, replace the corresponding $+$ in $\beta$ by $(e_{j})$.
\end{itemize}
\end{cor}

\begin{thm}\label{8.4.5}
 A Schubert variety $S_w$ which parametrized by a perfect major permutation $w$ has non-empty intersection with $2^{q}$ flag domains and the intersections with  the base cycles of these flag domains consist of $2^{2q}$ points.
\end{thm}
\begin{proof}
As in the proof of Theorem \ref{thm5.18} one shows that for each pair of the form $(-2j,(2j-1))$, for all $ 1\leq j\leq q$, there are $4$-possibilities for the vectors: $2$-possibilities with signature sign pair $(-+)$ and $2$-possibilities with signature sign pair $(+-)$. Since we have $q$ pairs, the Schubert variety $S_w$ intersects $2^q$ flag domains in $2^{2q}$ points.
\end{proof}

\begin{rem}
To describe the above intersection points between the base cycles and the Iwasawa Schubert variety $S_w$ of complimentary dimension we will define a set $Swit_w$ as follows:\\
  For a perfect major permutation $ w \in {\displaystyle S_{n}\rtimes \mathbb{Z}_2^{n}}$, define $Swit_w :=\{\varphi(w_r): w_r$ ~is~obtained~from~$w$~by~
  switching ~none,~some,~all~pairs  $(-2i,(2i-1))~by ~((2i-1),-2i)~or~(-(2i-1),2i)~or~(2i,-(2i-1))$, $1\leq i\leq q\} \subset W_T$.
 Define $\mathbb{F}(Fix~I)$ to be the set all maximal b-isotropic flags associated to the basis
\begin{eqnarray}
 e_1+e_{2q}, e_{2n-2q+1}+e_{2n},...,e_q+e_{q+1},e_{2n-q}+e_{2n-q+1},e_{2q+1},e_{2q+2},...,e_{n},e_{n+1},&&  \nonumber \\
 e_{n+2},...,e_{2n-q},e_q-e_{q+1},e_{2n-q}-e_{2n-q+1},...,e_1-e_{2q},e_{2n-2q+1}-e_{2n}.~~~~~~~~~& & \nonumber
\end{eqnarray}
Let $\textsl{F}_w\subset \mathbb{F}(Fix~T)$ be the set of all maximally $b$-isotropic flags associated to all elements in $Swit_w $. Note that for any perfect major permutation $ w$ we have $q$
pairs $(-2i,(2i-1)), 1\leq i\leq q $ for any $w \in W_I$, and for each pair we have 4 possibilities for switching it.  Hence the cardinality of $Swit_w $ is $4^q=2^{2q}$.
  The set $Swit_w $ gives us all intersection points of $S_w$ and each $2^{q}$ of these points belong to only one flag domain and are contained in the base cycle of that flag domain.
\end{rem}

\begin{ex}
 In $G_0=SP(6,4)$, fix $w=(-43-215)$ a perfect major permutation, then $Swit_w=\{(-34-215),(3-4-215),(4-3-215),(-43-215),(-432-15),\\(-43-125),(-431-25), (4-3-215),(4-32-15),(4-3-125),(4-31-25),\\(-34-215),(-342-15),(-34-125),(-341-25),(3-4-215),(3-42-15),(3-4-125),\\(3-41-25)\} $
\end{ex}
\begin{rem}
For the real form $SP(2p,2q)$, the number of Schubert varieties $S_w$ which parametrized by perfect major permutations  is
$$(n-1).(n-3)...(n-2q+1).$$
\end{rem}

\begin{rem}\label{8.4.11}
If $n=q+1$, then the flag domain parametrized by the sequence $$\alpha=+-+-...+-+-+$$ intersects all Schubert varieties of dimension $2pq$.
\end{rem}


\begin{center}

\end{center}

\begin{thebibliography}{99}
\vspace{.8cm}

\bibitem{F-S} F. Abu Shoga-Qarmout. Combinatorial geometry of flag domains in $G/Q$. Preprint in preparation.

\vspace{.4cm}

\bibitem{Bar} D. Barlet. Espace analytique réduit des cycles analytiques complexes compacts d'un espace analytique complexe de dimension finie. In François Norguet, editor, Fonctions de Plusieurs Variables Complexes II SE - 1, volume 482 of Lecture Notes in Mathematics, pages 1–158. Springer Berlin Heidelberg, 1975.

\vspace{.4cm}

\bibitem{BK} D. Barlet and V. Koziarz. Fonctions holomorphes sur l'espace des cycles: la méthode d'intersection. Mathematical Research Letters, 7:537–549,2000.

\vspace{.4cm}

\bibitem{Bre} A. Brecan. Schubert Slices in the Combinatorial Geometry of Flag Domains. PhD thesis, Jacobs University Bremen, expected Fall 2014.

\vspace{.4cm}

\bibitem{Bre2} A. Brecan. Schubert duality for $SL(n,\mathbb{R})$-flag domains. To appear.

\vspace{.4cm}

\bibitem{Bou} N. Bourbaki. "Groupes et algebras de Lie," Chaps. 4, 5, et 6, Hermann, Paris, 1968.

\vspace{.4cm}

\bibitem{FHW} G. Fels, A. Huckleberry, and J. A. Wolf. Cycle Spaces of Flag Domains, volume 245 of Progress in Mathematics. Birkhäuser-Verlag, Boston, 2006.

\vspace{.4cm}

\bibitem{FH} W. Fulton, J. Harris. Representation Theory: A First Course, Volume 129 of Graduate
texts in mathematics: Readings in mathematics, Springer study edition, 1991.

\vspace{.4cm}

\bibitem{Hel1} S. Helgason, Differential Geometry and Symmetric Spaces, Academic Press, New York 1962.

\vspace{.4cm}

\bibitem{Hel2} S. Helgason, Differential Geometry, Lie Groups and Symmetric Spaces, Academic Press, New York 1978.

\vspace{.4cm}

\bibitem{Hum} J. Humphreys. Linear Algebraic Groups, Volume 21 of Graduate Texts in Mathematics,
Springer, 1975.

\vspace{.4cm}

\bibitem{Kna} A. W. Knapp. Lie Groups Beyond an Introduction. Progress in Mathematics. Birkhäuser Boston, Boston, 2 edition, 2002.

\vspace{.4cm}




\bibitem{Onish} A.L.Onishik, E.B. Vinberg, Lie groups and algebraic groups, Springer Verlag, 1980.

\vspace{.4cm}

\bibitem{Onish} A.L.Onishik, E.B. Vinberg, Lie groups and algebraic groups, Springer Verlag, 1980.

\vspace{.4cm}

\bibitem{W1}  J. A. Wolf. The action of a real semisimple Lie group on a complex manifold, I: Orbit structure and holomorphic arc components, Bull. Amer. Math. Soc. 75 (1969), 1121–1237.

\vspace{.4cm}

\bibitem{W2}  J. A. Wolf. Exhaustion functions and co-homology vanishing theorems for open orbits on complex flag manifolds. Mathematical Research Letters,2(2), 1995.

\vspace{.4cm}

\bibitem{Yam} A. Yamamoto, Orbits in the flag variety and images of the moment map for classical groups I, Represent.
Theory 1 (1997) 327–404.


\end{thebibliography}
\end{document}